\documentclass[11pt, a4paper,reqno]{amsart}
\pdfoutput=1
\usepackage{tikz,enumitem, rotating,latexsym, bm,stmaryrd, caption,appendix}
\usetikzlibrary{positioning, decorations, shadings,shapes,arrows,shapes.geometric,calc,matrix,patterns}
\usepackage{enumitem}   
\usepackage{verbatim}

\definecolor{ao(english)}{rgb}{0.0, 0.5, 0.0}
\definecolor{darkgreen}{rgb}{0.0, 0.5, 0.0}

	\definecolor{eng}{rgb}{0.0, 0.5, 0.0}
\definecolor{apple}{rgb}{0.55, 0.71, 0.0}
\definecolor{cadmium}{rgb}{0.0, 0.42, 0.24}
\definecolor{darkspringgreen}{rgb}{0.09, 0.45, 0.27}
\definecolor{amethyst}{rgb}{0.6, 0.4, 0.8}
\definecolor{ao}{rgb}{0.0, 0.0, 1.0}
\definecolor{atomictangerine}{rgb}{1.0, 0.6, 0.4}
\definecolor{carmine}{rgb}{0.59, 0.0, 0.09}

\definecolor{toggle}{rgb}{1.0, 0.94, 0.96}

\usepackage[normalem]{ulem}

\def\down{\vee}
\def\up{\wedge}

\tikzset{
  variable line width/.style={
    every variable line width/.append style={#1},
    to path={%
      \pgfextra{%
        \draw[every variable line width/.try,line width=\pgfkeysvalueof{/tikz/thickness}] (\tikztostart) -- (\tikztotarget);
      }%
      (\tikztotarget)
    },
  },
  thickness/.initial=0.6pt,
  every variable line width/.style={line cap=round, line join=round},
}

\usepackage{thmtools}

 \usepackage{amsmath,amsthm,amsfonts,amssymb,mathrsfs,pb-diagram}
\usepackage[
bookmarks=true,colorlinks=true,linktoc=page,citecolor=green!80!black,linkcolor=green!80!black,urlcolor=green!80!black]{hyperref}
\usepackage{caption}
\usepackage{lipsum,wasysym}
\usepackage{mathtools}
\usepackage[a4paper,margin=2cm]{geometry}
\usepackage{cleveref}
 \usepackage{amsmath}
\mathchardef\mhyphen="2D
\usepackage{color}
\usepackage{xcolor}
\usepackage{ifthen}
\usepackage{sidecap}   
\definecolor{mediumblue}{rgb}{0.0, 0.0, 0.8}
 
\newcommand{\mptn}{{\mathscr{P}_{n}}}

\renewcommand{\geq}{\geqslant}
\renewcommand{\leq}{\leqslant}

\tikzset{wei/.style= 
{red,double=red,double
distance=0.5pt}}

\newcommand{\cp}{{\color{cyan}p}}
\newcommand{\ps}{{\color{purple}s}}
\newcommand{\gr}{{\color{darkgreen}r}}
\newcommand{\mq}{{\color{magenta}q}}
\newcommand{\ot}{{\color{orange}t}}

\newcommand{\vu}{{\color{violet}u}}
\newcommand{\ka}{\kappa}

\tikzset{wei2/.style={red,double=red,double
distance=0.5pt}}

\allowdisplaybreaks
\numberwithin{equation}{section}
\usepackage{scalefnt}

\newtheorem{thm}{Theorem}[section]

\newtheorem{lem}[thm]{Lemma}

\newtheorem{prop}[thm]{Proposition}

\newtheorem*{prop*}{Proposition}

\newtheorem*{thmA}{Theorem A}

\newtheorem*{cor*}{Corollary}

\newtheorem*{conj*}{Conjecture D}

\newtheorem*{conj1*}{Conjecture B}
\newtheorem*{Acknowledgements*}{Acknowledgements}

\newtheorem{rmk}[thm]{Remark}
\newtheorem{defn}[thm]{Definition}

\newcommand{\la}{\lambda}
\newcommand{\La}{\Lambda}

\newcommand{\ZZ}{{\mathbb Z}}

\let\<=\langle
\let\>=\rangle

\newcommand{\DP}{\underline{\mu}\la}

\crefname{ques}{Question}{Questions}
\crefname{defn}{Definition}{Definitions}
\crefname{thm}{Theorem}{Theorems}
\crefname{prop}{Proposition}{Propositions}
\crefname{lem}{Lemma}{Lemmas}
\crefname{cor}{Corollary}{Corollaries}
\crefname{conj}{Conjecture}{Conjectures}
\crefname{section}{Section}{Sections}
\crefname{subsection}{Subsection}{Subsections}
\crefname{eg}{Example}{Examples}
\crefname{figure}{Figure}{Figures}
\crefname{rem}{Remark}{Remarks}
\crefname{rmk}{Remark}{Remarks}
\crefname{equation}{equation}{equation}

\Crefname{ques}{Question}{Questions}
\Crefname{defn}{Definition}{Definitions}
\Crefname{thm}{Theorem}{Theorems}
\Crefname{prop}{Proposition}{Propositions}
\Crefname{lem}{Lemma}{Lemmas}
\Crefname{cor}{Corollary}{Corollaries}
\Crefname{conj}{Conjecture}{Conjectures}
\Crefname{section}{Section}{Sections}
\Crefname{subsection}{Subsection}{Subsections}
\Crefname{eg}{Example}{Examples}
\Crefname{figure}{Figure}{Figures}
\Crefname{rem}{Remark}{Remarks}
\Crefname{rmk}{Remark}{Remarks}

\usepackage[hang,flushmargin]{footmisc}

\begin{document}

 \title[Isotropic meta Kazhdan--Lusztig combinatorics II]{
Isotropic meta Kazhdan--Lusztig combinatorics II: Isomorphism to the generalised Khovanov arc algebra}

 \author{Ben Mills}
       \address{Department of Mathematics, 
University of York, Heslington, York,  UK}
\email{ben.mills@york.ac.uk}

\begin{abstract}
We construct an explicit isomorphism between the generalised Khovanov arc algebras of type D and the basic algebras of the anti-spherical Hecke category associated to the maximal parabolic subgroup $W (A_{n-1})$ of $W (D_n)$. This isomorphism maps generators to generators, thereby equipping the arc algebras with an Ext-quiver and relations presentation.
\end{abstract}

 \maketitle

  \vspace{-0.5cm}
   
  \section{Introduction}

   The ($p$)-Kazhdan--Lusztig polynomials feature prominently in a wide spectrum of problems in representation theory, geometry, and Lie theory. Given this ubiquity, in the companion to this paper \cite{My1}, we posed two questions regarding the isotropic form of these polynomials and whether the information they provide can be enriched to illuminate the structures they govern. We addressed the first question in \cite{My1} by enriching the type $(D_n, A_{n-1})$ decorated cup-cap algorithms, associating them with light leaves basis elements in the basic algebra of the anti-spherical Hecke category of isotropic Grassmannians, $\mathcal{H}_{(D_n,A_{n-1})}$. This association yielded an Ext-quiver and relations presentation of $\mathcal{H}_{(D_n,A_{n-1})}$.

 In this paper, we address the second question posed in \cite{My1}. Utilising the aforementioned presentation, we construct an explicit, elementary isomorphism between $\mathcal{H}_{(D_n,A_{n-1})}$ and the generalised Khovanov arc algebra of type $D$, denoted $\mathbb{D}_{\Lambda}$, as defined by Ehrig and Stroppel in \cite{TypeDKhov}. 
 
As remarked in \cite{TypeDKhov}, the algebra $\mathbb{D}_{\Lambda}$ is constructed from decorated cup diagrams, closely mirroring the diagrammatic framework of \cite{MR2813567,LEJCZYK_STROPPEL_2013} used to calculate the parabolic Kazhdan--Lusztig polynomials of type $(D_n, A_{n-1})$. Furthermore, building on the foundations of \cite{Soe07,elias2014hodgetheorysoergelbimodules}, Libedinsky and Williamson \cite{antiLW} established that these related Kazhdan--Lusztig polynomials can also be interpreted as composition multiplicities of standard modules in the anti-spherical Hecke category. We upgrade this shared combinatorial behaviour into an explicit algebraic equivalence in our main theorem: this proves that these two distinct structures, which are governed by the same Kazhdan--Lusztig polynomials, are equivalent. 

In the following, let $\Bbbk$ be an integral domain containing an element $i$ such that $i^2=-1$ and let $\La$ be a block of defect $k=\lfloor \frac{n}{2}\rfloor$ with parity matching $n$.

\begin{thmA}
\label{thm:A}
The generalised type $D$ Khovanov arc algebras $\mathbb{D}_\La$ are isomorphic (as $\mathbb{Z}$-graded $\Bbbk$-algebras) to the basic algebras $\mathcal{H}_{(D_n,A_{n-1})}$ of the anti-spherical Hecke categories associated to type $D_{n}$ with maximal parabolic of type $A_{n-1}$ for all $n\in\mathbb{N}$.
  \end{thmA}
This mirrors the results established with maximal parabolic of type $A_m\times A_n$ within type $A_{m+n}$ in \cite{ChrisDyckPaper}, from which this paper draws its inspiration.
 
The Khovanov arc algebra has its original applications in categorical knot theory \cite{Khov,Strop1}, and has since found profound applications across low-dimensional topology, geometry and representation theory \cite{Khovanov_2002,Brundan2010-mb,MR2918294,Brundan2011-ye,Brundan2012-uw,stroppel2022categorificationtangleinvariantstqfts}.
This construction was subsequently specialised to the type $D$, or isotropic, setting in a series of papers by Ehrig and Stroppel \cite{TypeDKhov,Ehrig_Stroppel_2016,ehrig2016schurweyldualitybraueralgebra,ehrig2017category,ehrig2018nazarovwenzlalgebrascoidealsubalgebras}. Although the resulting algebra, $\mathbb{D}_{\Lambda}$, is significantly more combinatorially intricate, it still boasts a wealth of applications. Notably, the category of modules over $\mathbb{D}_{\Lambda}$ is equivalent to the category of perverse sheaves on isotropic Grassmannians of type $D_n$ (and, via an isomorphism of varieties, type $B_{n-1}$) \cite{TypeDKhov}. Furthermore, bimodules over $\mathbb{D}_{\Lambda}$ correspond to projective functors in parabolic category $\mathcal{O}$ of type $(D_n, A_{n-1})$ and categorify the type $D$ Temperley–Lieb algebra of Green \cite{MR1618912}. Subsequent work linked $\mathbb{D}_{\Lambda}$ and type $D$ Springer fibres \cite{Ehrig_Stroppel_2016}. Additionally, specific idempotent truncations of $\mathbb{D}_{\Lambda}$ are isomorphic to blocks of cyclotomic Nazarov--Wenzl algebras \cite{ehrig2018nazarovwenzlalgebrascoidealsubalgebras}, and are related to the orthosymplectic Lie supergroup and Brauer algebras \cite{ehrig2015koszulgradingsbraueralgebras,ehrig2016schurweyldualitybraueralgebra,ehrig2017category,heidersdorf2024khovanovalgebrastypeb}

On the other hand, the aforementioned interpretation of Kazhdan--Lusztig polynomials as composition multiplicities of standard objects in the anti-spherical Hecke category led to the proof of the Kazhdan--Lusztig positivity conjecture \cite{antiLW} and a counter-example to Lusztig's famous conjecture \cite{williamson2018paritysheavesheckecategory}. This category plays a role in geometric representation theory, appearing in the description of perverse sheaves on partial flag varieties \cite{RW20,williamson2018paritysheavesheckecategory}, serving as a diagrammatic categorification of Hecke algebras \cite{MR3555156, Lib15, Wil18}, and featuring in the study of singular blocks of category $\mathcal{O}$ \cite{Soe90, Str04}. Additionally, its diagrammatic presentation, where morphisms are described via coloured graphs subject to local relations, facilitates the explicit construction of a light leaves basis and enables highly efficient computations.

We anticipate that Theorem A will enable the further study of related symmetric algebras and standard extension algebras in the manner of \cite{bowman2024quiverpresentationsschurweylduality,Strop-Eber}. Specifically, by translating problems into the Ext-quiver and relations presentation of $\mathcal{H}_{(D_n,A_{n-1})}$ detailed in \cite{My1}, this approach circumvents some of the diagrammatic complexity of $\mathbb{D}_\Lambda$. Working within $\mathcal{H}_{(D_n,A_{n-1})}$ has the distinct advantage of making the multiplication local, yielding straightforward formulae for computing the products of generators and allowing us to exploit its monoidal structure and to use further identities in the (easier to work with) light leaves basis to make subsequent calculations simpler.

  \subsection{Structure of this paper.}
 In \cref{recap} we review the meta Kazhdan--Lusztig style combinatorics from \cite{My1} that is needed for the Ext-quiver and relations presentation of $\mathcal{H}_{(D_n,A_{n-1})}$ and in \cref{MultiplicationIsHard}, we define the generalised Khovanov arc algebras in type $D$, $\mathbb{D}_\La$. In \cref{KhovContract} we describe a method for contracting larger diagrams in $\mathcal{H}_{(D_n,A_{n-1})}$ and $\mathbb{D}_\La$, which will allow us to restrict our attention to diagrams of smaller rank. Then in \cref{TheActualIso} we define the map in Theorem A and show that it is indeed a $\mathbb{Z}$-graded isomorphism of $\Bbbk$-algebras.

  \section{Oriented cup combinatorics for the Hecke category}
  \label{recap}
     
 We start by reviewing the Kazhdan--Lusztig style combinatorics that is needed for the presentation of $\mathcal{H}_{(D_n,A_{n-1})}$  in \cite{My1}. We will later utilise this presentation to construct our isomorphism. Much of what is detailed in this section can be found in more generality in \cite{ChrisHSP,bowman2023orientedtemperleyliebalgebrascombinatorial}, and we refer to \cite{MR2813567,LEJCZYK_STROPPEL_2013} for the origin of this type $D$ cup-cap combinatorics. But in this section, we include only what we need for this paper, namely, type $D$ Kazhdan--Lusztig style oriented cup-cap combinatorics.
     
    \subsection{Weights and  Kazhdan--Lusztig polynomials }
    Let $W(C_n)$ denote the group of signed permutations of $\{1,\dots,n\}$  generated by  the 
    set of permutations $\{s_{0'},s_1,s_2,\dots , s_{n-1}\}$ with Coxeter relations encoded in the leftmost diagram of \cref{coxeterlabelD2}.
    We let  $W(D_n)$ denote the subgroup of even signed permutations of $\{1,\dots,n\}$  generated by  the 
    set of permutations $\{s_{0},s_1,s_2,\dots , s_{n-1}\}$ where $s_0=s_{0'}s_1s_{0'}\in C_n$, with Coxeter relations encoded in the rightmost diagram of \cref{coxeterlabelD2}. 
Lastly, we let $W(A_{n-1})$  denote the subgroup (of $W(D_n)\leq W(C_n)$)   generated by the reflections $\{s_1, s_2,....., s_{n-1}\}$, coloured white in \cref{coxeterlabelD2}.
     
 \begin{figure}[ht!]
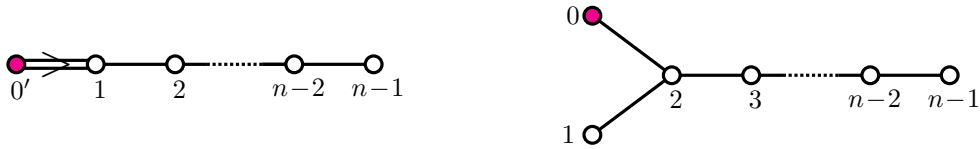

 \begin{minipage}{6cm}

 \end{minipage}
 \caption{ The Dynkin diagrams for the Weyl groups $W(C_n)$ and $W( D_n)$ with the labelling to be used throughout this paper, with the node not belonging to the parabolic of type $W(A_{n-1})$ highlighted in pink.}
\label{coxeterlabelD2}
\end{figure}

 We define a \textsf{weight} to be a horizontal strip with $n$ vertices at half-integer $x$-coordinates $\{\frac{1}{2},\frac{3}{2},...,n-\frac{1}{2}\}$, each labelled by either $\up$ or $\down $. 
 We let $\{s_i \mid 1 \leq i \leq n\}$ act by permuting the labels at $i-\frac{1}{2}$ and $i+\frac{1}{2}$ and fixing every other node;
we let $s_{0'}$ flip the label at the 1st node, at $x=\frac{1}{2}$, through its horizontal axis (that is $s_{0'}(\vee) = \wedge$).  
In this manner, the generator $s_{0}=s_{0'}s_1s_{0'}$ can be seen as acting 
first by flipping the label at $x=\frac{1}{2}$ and $x=\frac{3}{2}$ through the horizontal axis and then permuting their order 
(that is, $s_{0}(\vee\vee) = \wedge \wedge$,  $s_{0}(\up\up) = \down\down$, $s_{0}(\vee\wedge) = \vee\wedge$ and $s_{0}( \wedge\vee) =  \wedge\vee$).

 For a parabolic Coxeter system $(W,P)$, we denote by ${^PW}$ the set of minimal length right coset representatives of $P\setminus W$. The action on weights described above allows us to label the entire set of cosets of ${^PW}$  for $P=W(A_{n-1})$ and $ W=W(C_n)$
  via the set of all weights.
  Similarly, we can encode a set of minimal length coset representatives of ${^PW}$  for $P=W(A_{n-1})$ and $W=W(D_n) $ 
  via the subset of all weights 
   with an {\em even} number of vertices labelled $\up$.
Specifically, we set the identity weight,  $\emptyset$, to be the weight consisting of  $n$ $\vee$'s. 
Then the other ($2^{n-1}$) cosets can be obtained by permuting the labels of the identity weight.
With this identification in place, we shall, for the remainder of this paper, set $(W,P)=(W(D_n),W(A_{n-1}))$ with generators as above.

We can also visualise the elements of $^PW$ as tilings within a restricted admissible region: $$\mathscr{A}_{(W,P)}\coloneqq\{ [r,c] \mid r,c \leq  n \text{ and }r-c\geq 0 \},$$ where we draw these tilings ``Russian Style" with rows (fixed values of $r$) pointing northwest and columns (fixed values of $c$) pointing northeast.
We say that a pair of tiles  
  are {\sf neighbouring} if they meet at an edge.
 Given a pair of neighbouring tiles $X$ and $Y$, we write $Y < X$ if $X$ appears above $Y$ (i.e. the $y$-coordinate of $X$ is strictly larger than that of $Y$). This extends to a partial order on all tiles in $\mathscr{A}_{(W,P)}$ and we say $Y$ {\sf supports} $X$ if $Y<X$ in this ordering. 
 We say that a collection of tiles, $\la\subseteq \mathscr{A}_{(W,P)}$ is a
  {\sf tile-partition} if every tile in $\la$ is
   supported and we denote the set of tile-partitions $\mptn$.  
For $\lambda  ,\mu  \in \mptn$, we define the {\sf Bruhat order} on tile-partitions by
 $\lambda  \leq \mu $ if  
$$\{[r,c] \mid [r,c] \in \lambda \}
\subseteq \{[r,c] \mid [r,c] \in \mu \}.$$

 \begin{figure}[ht!]
 \centering


\caption{
We depict the 
 identity weight $\emptyset$ along the bottom of the diagram, 
  the weight of $\mu=(1,2,3,4,5,3,3)$ along the top of the diagram,    
  and the tile-partition $\mu$ which corresponds to the coset labelled by the reduced word
  $\color{gray}s_0
    \color{red}s_2  
      \color{blue}s_3
          \color{cyan}s_4
          \color{magenta}s_5
  \color{lime!80!black}s_6
      \color{yellow}s_7  
    \color{darkgreen}s_1
        \color{red}s_2  
      \color{blue}s_3
          \color{cyan}s_4
          \color{magenta}s_5
  \color{lime!80!black}s_6
      \color{gray}s_0
    \color{red}s_2  
      \color{blue}s_3
          \color{cyan}s_4
          \color{magenta}s_5
            \color{darkgreen}s_1
        \color{red}s_2  
         \color{gray}s_0$.
  }
\label{typeAtiling-long}
\end{figure}

\begin{rmk}
We have now labelled a coset, $^PW$, by both a weight diagram and a tile-partition. We now detail the specific bijection between the two labellings.  

To convert a weight diagram to a tile partition, we read the labels of the weight diagram from right to left. Beginning at the rightmost corner of the admissible region, trace a north-western step for each $\up$ label and a south-western step for each $\down$ label. After $n-1$ steps, we reach the left-hand ``jagged" wall of the admissible region, having traced the northern perimeter of the corresponding tile-partition. In particular, the identity coset corresponds to the weight diagram consisting of $n$ consecutive $\down$ labels, which traces the boundary of the empty tile partition $\emptyset$. This process can be reversed in the obvious manner.
\end{rmk}

  \begin{rmk}There is a natural bijection between ${^P}W$ and 
$\mptn$ (see   \cite[Appendix]{MR3363009}) under which the length functions coincide. 
\end{rmk}

For conciseness and to explicitly track the parameter $n$, we will index each $\lambda \in {^P}W$ by its corresponding tile partition $\lambda \in \mptn$ (written as a sequence of row lengths). However, for the majority of our practical calculations and arguments, we will work directly with their weight diagrams.

Next we recall the construction of  the cup diagram $\underline{\mu}$ associated to a weight $\mu\in \mptn$ as defined originally
 in \cite{MR2813567,TypeDKhov}.

  \begin{defn}[Decorated cup diagram]
  \label{drawcups}
 For each weight $\mu\in\mptn$ we can construct the (decorated) {\sf cup diagram} $\underline{\mu}$ 
 using the following algorithm.
 \begin{enumerate}[leftmargin=*,label=(\roman*)]
\item Find a pair of vertices labelled $\down$ $\up$ in order from left to right that are neighbours in the sense that there are only vertices already joined by cups at an earlier stage in between. Join these new vertices together with an undecorated cup.
\item Repeat this process until there are no more such $\down$ $\up$ pairs. 
\item Then, starting from the left, connect neighbouring vertices labelled with $\up$ $\up$ with a decorated cup, which we will denote by drawing a $\bullet$ in the centre of the cup.
\item Finally, draw an infinite line vertically downwards at all the remaining $\up$ and $\down$ vertices. If there exists a (necessarily unique) $\up$ vertex remaining, then decorate the corresponding propagating line, drawing the $\bullet$ in the centre of the line. We refer to these vertical lines as {\sf rays}.
\end{enumerate}
\end{defn}

  \begin{figure}[ht!]
  $$   
$$
\caption{The construction of the cup diagram $\underline{\la}$ for $\la$ as in \cref{drawcups}.
	 See also \cite[Proposition 7.1]{bowman2023orientedtemperleyliebalgebrascombinatorial}.}
\label{construction}
	\end{figure}

  \begin{defn}
  \label{oriented or not?}
  For any $\la,\mu \in \mptn$, we can form a new diagram $ \underline{\mu} \la$ by glueing the weight $\la$ on top of $\underline{\mu}$. 
We say that $ \underline{\mu} \la$ is an {\sf oriented cup diagram} if 
\begin{itemize}
\item The vertices at either end of any undecorated cup have one of the labels pointing into the cup and the other pointing out of the cup.
\item  The vertices at either end of any decorated cup either both point in or both point out of the cup.
\item The vertex at the end of any undecorated ray is $\down$, and the vertex at the end of any decorated ray is $\up$.
\end{itemize}

For each strand $S$ in $\underline{\mu} \la$, we define \textsf{$l_S$} to be the $x$-coordinate of the leftmost vertex of $S$ and \textsf{$r_S$} to be the $x$-coordinate of the rightmost vertex of $S$. Recall that we have defined weights to have labels at the half-integer coordinates so $l_s, r_s \in \frac{1}{2}+\mathbb{Z}$.

 \end{defn}
 
 Some examples are depicted in \cref{removal} where the highlighted cup $\cp\in\underline{\mu}$ has $l_{\cp}=\frac{7}{2}$, $r_\cp=\frac{13}{2}$.
 Next, we define the breadth of a cup, a measure of its size, with a distinction between decorated and undecorated cups. For the reason for this distinction, see \cite[Section 2]{My1} or for an alternative point of view, \cite[Section 1.3]{TypeDKhov} and \cite[Subsection 4.2]{LEJCZYK_STROPPEL_2013}.
 
  \begin{defn}
Let $\mu\in \mptn$. Given a cup $\cp  \in \underline{\mu}$ we define its {\sf breadth}, $b(\cp)$, as follows:
 
If $\cp$ is undecorated, then $b(\cp)=\frac{1}{2}(r_\cp-l_\cp+1)$.
  
  If $\cp$ is decorated, then $b(\cp)=\frac{1}{2}(r_\cp+l_\cp)$
 \end{defn}

\begin{defn}
\label{degrees}
   Let $\lambda, \mu$ be weights such that $\underline{\mu}\lambda$ is oriented. We define the {\sf degree} of the diagram $\underline{\mu}\lambda$ 
   to be the number of cups whose right vertex is labelled by $\down$ 
in the diagram, that is
\[\deg(\underline{\mu}\lambda)=
		\sharp\left\{  \begin{minipage}{1.2cm}

		\end{minipage}
		\right\}.	 \]
We also refer to these cups, whose right vertex is labelled by $\down$, as being {\sf clockwise oriented}. Similarly, we refer to cups, whose right vertex is labelled by $\up$, as being {\sf anti-clockwise oriented}. In this sense, we can think of a decoration on a cup as an orientation-reversing point.
 \end{defn}

 \begin{figure}[ht!]
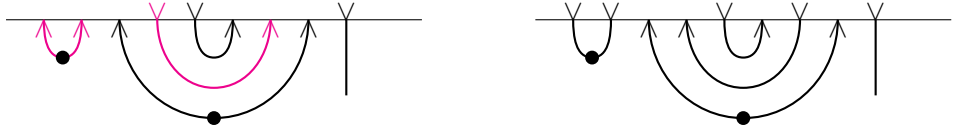

 $$
 %
$$
\caption{
  The degree of $\underline{\mu}\la$ on the right is two, which is precisely the number of (highlighted) flipped cups from $\underline{\mu}\mu$ on the left.
  }
\label{flipit}
\end{figure}

For the purposes of this paper, for $p\geq 0$, we define the $p$-Kazhdan--Lusztig polynomials of type $(W,P) = (D_n, ,A_{n-1})$ as follows.  
For $\la,  \mu \in \mptn $ we set
$$
{^p}n_{\la,\mu}= 
\begin{cases}
q^{\deg(\underline{\mu} \la)}		&\text{if $ \underline{\mu} \la $ is oriented}\\
0						&\text{otherwise.}
\end{cases}
$$
We refer to \cite[Theorem 7.3]{bowman2023orientedtemperleyliebalgebrascombinatorial} 
and \cite[Theorem A]{ChrisHSP} 
for a justification of this definition and to \cite{MR2813567,MR2918294} 
for the origins of this combinatorics.

\begin{rmk}
\label{flipper}
We say we {\sf flip} a cup if we reflect both the vertices of the cup through the horizontal axis so as to change the orientation of the cup from anti-clockwise to clockwise or vice versa.
 It is clear that for a fixed $\mu\in \mptn $, the diagram $\underline{\mu}\lambda$ is oriented of degree $d$ if and only if the weight $\lambda$ is obtained from the weight $\mu$ by flipping $d$ anti-clockwise cups in $\underline{\mu}$. There are precisely 2 ways of doing this: either flipping undecorated or decorated anti-clockwise cups. Pictorially, this looks like
$$
(i)\begin{minipage}{1.2cm} 
\end{minipage}.
$$
 Moreover, the degree of $\underline{\mu}\la$ is precisely the total number of such swapped pairs. 
See \cref{flipit} for an example of a cup diagram of degree two.
\end{rmk}

Let $\mu,\lambda \in \mptn$ be such that $\la$ is obtained from $\mu$ by flipping a single cup $\cp \in \underline{\mu}$. Then we say $ \lambda=\mu - \color{cyan}{p}$. An example is below in \cref{removal}. We also write $\mu = \la +\cp$.
 
  \begin{figure}[ht!]
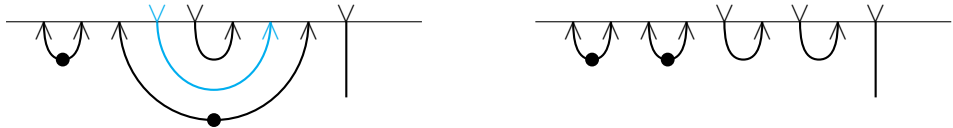

 $$
 %
$$
\caption{
  The flipping of the vertices of the cup $\cp$ from $\underline{\mu}$ creates the weight $\la$. Hence we say $\la=\mu- \cp$, where we have drawn $\underline{\la}$ using \cref{drawcups}
  }
\label{removal}
\end{figure}

      \subsection{ Cup combinatorics}
   \label{Cup combinatorics}
We now set up the meta--Kazhdan--Lusztig combinatorics of these oriented cup diagrams that govern the relations of $\mathcal{H}_{(D_n,A_{n-1})}$.

\label{dyckgens}

 \begin{defn} 
\label{non-com:def} 
 Let $\mu \in \mptn$ and let $\cp,\mq \in \underline{\mu}$.
 We say that $\cp$ {\sf covers} $\mq$ and we write $\mq$ $\prec \cp$ if and only if 
    $l_{\cp} < l_{\mq}$ and $r_{\cp} > r_{\mq}$. See \cref{covering1,covering2} for examples.
    \end{defn}

    \begin{figure}[ht!]
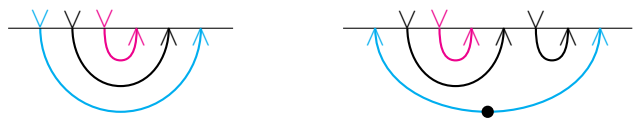

  $$   
$$
		 \caption{Examples of commuting pairs $\cp$ and $\mq$ such that $\mq \prec \cp$.}
 \label{covering1}
\end{figure}  
\begin{figure}[ht!]
$$
 %
$$
 \caption{Examples of non-commuting pairs $\cp$ and $\mq$ such that $\mq \prec \cp$.}
 \label{covering2}
\end{figure}

    \begin{defn}
\label{doublenon-com:def} 
 Let $\mu \in \mptn$ and let $\cp,\mq \in \underline{\mu}$.
 We say that $\cp$ {\sf doubly covers} $\mq$ and write $\mq \prec \prec \cp$ if and only if $\color{cyan}q$ is decorated and $ r_\mq< l_\cp$.  See \cref{doublecover2,doublecover1} for examples.
\end{defn}

    \begin{figure}[ht!]
     $$   %
	$$
		 \caption{Examples of commuting pairs $\cp$ and $\mq$ such that $\mq \prec \prec \cp$.}
 \label{doublecover1}
\end{figure}  
\begin{figure}[ht!]

	$$  %
$$ 
    
    \caption{Examples of non-commuting pairs $\cp$ and $\mq$ such that 
    $\mq \prec \prec \cp$.}
    \label{doublecover2}
    \end{figure}

\color{black}
Next, we present a rather cumbersome definition of commuting cups, but for the intuition, one should simply skip straight to \cref{commie2}.
       \begin{defn}\label{dist:def}
  Let $\mu\in\mptn$ and $\cp, \mq \in \underline{\mu}$ and suppose without loss of generality that $l_{\mq} < l_\cp$. 
  We say that $\color{magenta}p$ and $\color{cyan}q$  commute if 
  either
$(i)$ there exists $\gr $ such that 
$\mq \prec 	\gr		\prec \cp$,
   $(ii)$ there exists $\gr $ such that 
$\mq \prec 	\gr		\prec\prec \cp$ 
or 
$\mq \prec\prec 	\gr		\prec\prec \cp$ 
or $(iii)$ neither  $\color{magenta}p$  nor     $\color{cyan}q$ (doubly) covers the other.  
  \end{defn}  
  
     Above 
 \cref{covering1,doublecover1} consist of commuting pairs of cups 
 and  \cref{covering2,doublecover2} consist of (doubly) non-commuting pairs of cups.
  
\begin{rmk}
\label{commie2}
Let $\mu\in\mptn$ and $\cp,\mq\in \underline{\mu}$. Then $\cp$ and $\mq$ commute if and only if $\cp \in \underline{\mu-\mq}$ and $\mq \in \underline{\mu-\cp}$. This will follow naturally from the following lemmas and \cref{general}.
\end{rmk}

   \begin{defn} \label{adj:def2}
   Let $\mu, \la \in \mptn$ be such that $\la=\mu-\cp$. Let $\cp$ $\in \underline{\mu}$ and $\ot\in \underline{\la}$.
 We say that $\cp$ and $\ot$ are {\sf adjacent} if $\ot$ has precisely one node in common with $\cp$.
\end{defn}
Clearly, a cup $\cp\in \underline{\mu}$ can be adjacent to zero, one or two cups in $\underline{\la}$. We will list a complete set of cases below, which follows immediately from \cref{adj:def2}, but can also be deduced from \cite[Lemma 3.9]{TypeDKhov}.

\begin{lem} \label{adj lem0}
Let $\la,\mu\in \mptn$ be such that $\la=\mu-\cp$. Suppose $\cp$ is not (doubly) covered in $\underline{\mu}$. Moreover, if there is no undecorated strand to the left of $\cp$ and no decorated strand to the right of $\cp$, then $\cp$ is adjacent to zero cups in $\underline{\la}$.  The following is a complete list of the pairs for $\cp\in \underline{\mu}$ and $\underline{\la}$ that can be pictured on two vertices:

    $$ a) \Biggl( 
    \begin{minipage}{1.4cm}
\end{minipage}\Biggr)$$

 \end{lem}

\begin{lem} \label{adj lem1}
Let $\la,\mu\in \mptn$ be such that $\la=\mu-\cp$. If $\cp$ is not (doubly) covered, then $\cp$ can be adjacent to at most one cup $\ot$. The following is a complete list of the adjacent pairs $\cp\in \underline{\mu}$ and $\ot \in \underline{\la}$ that can be pictured on three vertices:
    $$ a)  \Biggl( 
    \begin{minipage}{1.9cm} %
  \end{minipage}\Biggr) 
	 $$

 \end{lem}          
 
\begin{lem} \label{adj lem2}
Let $\la,\mu\in \mptn$ be such that $\la=\mu-\cp$. If $\cp$ is (doubly) covered in $\underline{\mu}$, then $\cp$ is adjacent to two cups, $\gr, \ot \in \underline{\la}$. The following is a complete list of the pairs $\cp\in \underline{\mu}$ and $\gr,\ot \in \underline{\la}$ that can be pictured on four vertices:
 \begin{itemize}[leftmargin=*] 
\item   Firstly, we consider the case where $\cp$ is covered by a non-commuting cup, $\vu\in \mu$ (we additionally say that $\cp$ is {\sf covered}):
  $$   
a) \Biggl( 
    \begin{minipage}{2.4cm}  
 \end{minipage}\Biggr).$$ Notice that the two cups $\gr, \ot$ on the right-hand side are commuting pairs and the total number of decorations is the same in $\underline{\mu}$ and $\underline{\la}$ in both cases $a)$ and $b)$.

\item Secondly, we consider the 
case that  $\cp$ is doubly covered by a non-commuting decorated cup, $\vu \in\underline{\mu}$  (we additionally say that $\cp$ is {\sf doubly covered}):

    $$  c)\Biggl( 
    \begin{minipage}{2.4cm} %
 \end{minipage}\Biggr).$$ 
	 Notice that the two cups $\gr, \ot$ on the right-hand side are non-commuting pairs with $\gr \prec \ot$.
\end{itemize}

 \end{lem}
 
In \cref{adj lem2} there is a (purple) cup $\vu\in\underline{\mu}$ which, due to the flipping of the vertices of $\cp$, no longer appears in $\underline{\nu}$ and shares one node each with $\gr$ and $\ot$. This is the smallest possible cup (doubly) covering $\cp$. As we will argue below, this cup is uniquely defined even for more general cases that need $n>4$ vertices to be pictured.
 
 \begin{defn} \label{gen def}
   Let $\mu, \la \in \mptn$ be such that $\la=\mu-\cp$. Let $\cp$ $\in \underline{\mu}$ and $\ot \in \underline{\la}$ be adjacent.
Suppose $\vu\in \underline{\mu}$, $\cp \neq \vu$ is the cup such that $l_\vu=l_\ot$ or $r_\vu=r_\ot$. Then we say that $\vu$ is the \textsf{cup generated by $\cp$ and $\ot$}, denoted $\langle {\color{cyan}p}\cup\ot\rangle_\mu$.
   \end{defn}

We should note that we are intentionally not defining $\langle {\color{cyan}p}\cup {\color{darkgreen}r}\rangle_\mu$ in the case where then $\cp$ is adjacent to a non-commuting pair $\gr, \ot$ of cups and $\gr \prec \ot$ - in this case $l_\vu$=$r_\gr$ or $r_\vu$=$l_\gr$. 
      
\begin{rmk}
\label{general}

The cases covered in \cref{adj lem2,adj lem1,adj lem0} are the ones pictured on the minimal number of vertices of $\underline{\mu}$. One can create equivalent cases on a larger number of vertices by adding (possibly many) cups that are (doubly) covered by $\vu$ and/or $\cp$. However, each of these cases reduces to one of the cases above. For example,
$$ 
 \Biggl( 
    \begin{minipage}{4.3cm}  
 \end{minipage}\Biggr).$$is equivalent to \cref{adj lem2} b) and 
	$$  
	 \Biggl( 
    \begin{minipage}{5.2cm}  %
 \end{minipage}\Biggr)$$is equivalent to \cref{adj lem2} c). One can similarly generalise the remaining cases in \cref{adj lem0,adj lem1,adj lem2}.

In \cref{KhovContract} we will explain a method that will allow us to move between oriented cup diagrams, Hecke categories and Khovanov arc diagrams of different sizes, 
which will allow us to restrict our focus as much as possible to the cases that correspond to the cup diagram with the minimal number of vertices, as pictured in \cref{adj lem0,adj lem1,adj lem2}. 
\end{rmk}

The final lemma in this section is a simple rephrasing of \cref{flipper}.
\begin{lem}
Let $\la,\mu, \nu \in \mptn$ and suppose $\underline{\mu}\la$ is oriented with $\la= \mu - \sum_{i=1}^{m}\cp^i$, $\cp^i \in \underline{\mu}$. Suppose for some $\ot \in \underline{\la}$ that $\nu= \la - \ot$, then  $\underline{\mu}\nu$ is oriented if and only if $\ot$ is not adjacent to any $\cp^i$.
\end{lem}

 \subsection{The presentation of $\mathcal{H}_{(D_n, A_{n-1})}$}
 
 For a general maximal parabolic, the algebra $\mathcal{H}_{(W,P)}$ originates from the groundbreaking framework of Soergel bimodules and their diagrammatic categorification of the Hecke algebra \cite{Elias_2010}. Within this categorical setting, one frequently leverages the Libedinsky–Williamson construction of the light leaves basis to perform explicit calculations \cite{MR3555156,antiLW,ChrisHSP}.

We now introduce the first of the two algebras in this paper: the basic algebra of the anti-spherical Hecke category associated to the maximal parabolic $(W,P)=(W(D_n),W(A_{n-1}))$, and also recall key properties about its structure. 
For a general maximal parabolic, $\mathcal{H}_{(W,P)}$ was developed in the framework in Soergel bidmoules and their diagrammatic categorification of the Hecke algebra \cite{Soe07,Elias_2010}. Within this setting, one will often use Libedinsky-Williamson's construction of the light leaves basis for explicit calculations \cite{MR3555156,antiLW,ChrisHSP}.
Instead of that approach, we make use of the main result of \cite{My1}, which established a quadratic presentation of $\mathcal{H}_{(D_n, A_{n-1})}$ defined entirely in terms of the oriented cup combinatorics outlined in \cref{Cup combinatorics}, based on the work in \cite{bowman2023orientedtemperleyliebalgebrascombinatorial,ChrisDyckPaper}. In \cref{TheActualIso} we will make use of this presentation to construct an isomorphism to the generalised Khovanov arc algebra. For the purposes of this paper we define  $\mathcal{H}_{(D_n, A_{n-1})}$ as follows:
 
 \begin{thm}\cite[Theorem 5.9]{My1}
 \label{presentation}
The algebra $\mathcal{H}_{(D_n, A_{n-1}) }$ is the associative $\Bbbk$-algebra generated by the elements 
\begin{equation}\label{geners}
\{D^\la_\mu,
D_\la^\mu \mid 	
\text{$\la, \mu\in \mptn $ with $\la = \mu - {\color{cyan}p}$ for some ${\color{cyan}p}\in \underline{\mu}$} 
	\}\cup\{ {\sf 1}_\mu \mid \mu \in \mptn  \}		
	\end{equation}
	subject to the following relations and their duals. 

	\smallskip\noindent
{\bf The idempotent   
relations:} 
For all $\la,\mu \in \mptn $, we have that 
\begin{equation}\label{rel1}
{\sf 1}_\mu{\sf 1}_\la =\delta_{\la,\mu}{\sf 1}_\la \qquad 
\qquad {\sf 1}_\la D^\la_\mu {\sf 1}_\mu = D^\la_\mu.
\end{equation} 

\smallskip\noindent
{\bf The commuting relations:} 
Let ${\color{cyan}p},{\color{magenta}q}\in \underline{\mu}$ which commute. Then we have 
\begin{equation}\label{commuting}
D^{\mu -{\color{cyan}p}-{\color{magenta}q}}_{\mu-{\color{cyan}p}}D^{\mu-{\color{cyan}p}}_\mu = D^{\mu-{\color{cyan}p}-{\color{magenta}q}}_{\mu -{\color{magenta}q}}D^{\mu -{\color{magenta}q}}_\mu \qquad
D^{\mu -{\color{cyan}p}}_\mu D^\mu_{\mu - {\color{magenta}q}} = D^{\mu - {\color{cyan}p}}_{\mu - {\color{cyan}p} - {\color{magenta}q}}D^{\mu - {\color{cyan}p} - {\color{magenta}q}}_{\mu - {\color{magenta}q}}.
\end{equation}

\smallskip\noindent
{\bf The non-commuting relation:}  
Let $\cp,\mq \in \underline{\mu}$ with $\mq \prec \cp $ be non-commuting. 
Then $\mq$ is adjacent to a pair of commuting non-concentric cups, which 
we label by  ${\color{darkgreen}q}^1$ and ${\color{orange}q}^2$ from left to right in $\underline{\mu - \mq}$ (with ${\color{darkgreen}q}^1$ possibly decorated). 
Then we have:
\begin{equation}\label{noncommutingcup1}
D^{\mu -\cp}_\mu D^\mu_{\mu-\mq} = 
 D^{\mu - \cp}_{\mu - \cp - {\color{darkgreen}q}^1}D^{\mu - \cp - {\color{darkgreen}q}^1}_{\mu - \mq} =
 D^{\mu - \cp}_{\mu - \cp - {\color{orange}q}^2}D^{\mu - \cp - {\color{orange}q}^2}_{\mu - \mq}.
\end{equation}

\smallskip\noindent
{\bf The doubly non-commuting relation:}  
Let $\cp,\mq \in \underline{\mu}$ with $\mq \prec \prec \cp $ be doubly non-commuting. 
Then $\mq$ is adjacent to a pair of concentric non-commuting cups, 
 ${\color{darkgreen}q}^1$ and ${\color{orange}q}^2$ with ${\color{darkgreen}q}^1  \prec {\color{orange}q}^2$ in $\underline{\mu - \mq}$ (with ${\color{orange}q}^2$ possibly decorated).
Then we have: 
\begin{equation}\label{noncommutingcup2}
D^{\mu -\cp}_\mu D^\mu_{\mu-\mq} = D^{\mu -\cp }_{\mu - \cp - {\color{orange}q}^2}D^{\mu -\cp- {\color{orange}q}^2}_{\mu -\mq}   
\end{equation}

\smallskip\noindent
	{\bf The 
	self-dual relation: } 
	Let  ${\color{cyan}p}\in \underline{\mu}$ and $\la = \mu - {\color{cyan}p}$. If $\cp$ is doubly covered and hence the pair $(\ot, \gr)$ adjacent to $\cp$ is non-commuting with $\gr \prec \ot$, then we have
	\begin{equation}
D_\mu^{\la} D_{\la}^\mu
= (-1)^{b({\color{cyan}p})-1}\Bigg(
2
\!\! \sum_{   \begin{subarray}{c} \mq \in (\underline{\mu} \cap \underline{\la}) \\ \cp \prec\mq,\cp \prec\prec\mq \end{subarray}}
\!\!
(-1) ^{b({\color{magenta}q}) } D^{\la}_{\la- {\color{magenta}q} } D^{\la - {\color{magenta}q}  }_{\la} + 
2  (-1)^{b(\ot)} D^{\la}_{\la- \ot } D^{\la -\ot }_{\la} + 
 (-1)^{b(\gr)} D^{\la}_{\la- \gr } D^{\la - \gr }_{\la}\Bigg)
 \label{selfdualrelD}
\end{equation}
and otherwise we have 
\begin{equation}
D_\mu^{\la} D_{\la}^\mu
= (-1)^{b({\color{cyan}p})-1}\Bigg(
2
\!\! \sum_{   \begin{subarray}{c} \mq \in (\underline{\mu} \cap \underline{\la}) \\ \cp \prec\mq,\cp \prec\prec\mq \end{subarray}}
\!\!
(-1) ^{b({\color{magenta}q}) } D^{\la}_{\la- {\color{magenta}q} } D^{\la - {\color{magenta}q}  }_{\la} + 
\!\!\sum_{  \begin{subarray}{c} \gr \in \underline{\la} \\ \gr \, \text{adj.}\, {\color{cyan}p} \end{subarray}}
\!\!
 (-1)^{b(\gr)} D^{\la}_{\la- \gr} D^{\la - \gr }_{\la}\Bigg), 
 \label{selfdualrelA}
\end{equation}  
where throughout we refer to the set $\underline{\mu} \cap \underline{\la}$ to be the cups in $\underline{\mu}$ that commute with $\cp$ (and hence are in $\underline{\la}$ also) and abbreviate ``adjacent to" simply as ``adj."

\smallskip\noindent
{\bf The adjacent relation}  
Given $\la=\mu-\cp$,  suppose that  ${\color{cyan}p}\in \underline{\mu}$ and $\ot\in \underline{\la}$  are  adjacent.
Then we have:
\begin{equation}\label{adjacentcup}
D^{\mu - {\color{cyan}p} - \ot}_{\mu - {\color{cyan}p}}D^{\mu - {\color{cyan}p}}_\mu = 
\left\{ \begin{array}{ll} (-1)^{b(\langle {\color{cyan}p}\cup \ot\rangle_\mu) - b(\ot)} 
D^{\mu - {\color{cyan}p} - \ot}_{\mu - \langle {\color{cyan}p}\cup \ot\rangle_\mu}
D^{\mu - \langle {\color{cyan}p}\cup \ot\rangle_\mu}_\mu &
 \mbox{if $\langle {\color{cyan}p}\cup \ot\rangle_\mu$ exists}
  \\ 0 
& \mbox{otherwise} \end{array} \right.
\end{equation}
 
 \end{thm}

 The following property is due to Libedinsky-Williamson \cite{antiLW}, but it is phrased in the language of this paper in \cite[Theorem 5.5]{My1}.

 \begin{thm}(\cite{antiLW})
 \label{cellular basis}
 The  algebra   
$\mathcal{H}_{(D_n,A_{n-1})}  $  is a basic graded cellular
algebra with a graded cellular basis given by
\begin{equation}\label{basis}
\{D_\la^\mu D^\la_\nu \mid  \la,\mu,\nu\in \mptn  \,\, \mbox{ such that }\,\,		\underline{\mu}\la,\underline\nu\la \text{ are oriented} \}
\end{equation}
with 
$${\rm deg}( D_\la^\mu D^\la_\nu) = {\rm deg}(\underline{\mu}\la) + {\rm deg}(\underline{\nu}\la),$$ with respect to the involution $(D_\la^\mu)^*=D_\mu^\la$ and the partial order on $\mptn $ given by the Bruhat order.
\end{thm}
 
   \section{The generalised Khovanov arc algebra of type $D$}
   \label{MultiplicationIsHard}

   In this section, we review the generalised Khovanov arc algebra of type $D$, $\mathbb{D}_{\La}$, first constructed in \cite{TypeDKhov}. This original construction is combinatorially intensive; indeed, a significant portion of that paper is dedicated to showing that multiplication in $\mathbb{D}_{\La}$ is associative, something that is far from obvious. The multiplication is highly non-local, and we hope that one of the benefits of the work in this paper will be to provide a tractable framework for the structure of $\mathbb{D}_{\La}$. For the sake of brevity, we recall only the definitions strictly necessary for this paper and defer the justification for this to \cref{myexcuse,only two,get out of jail remark}. We note that our terminology differs slightly from \cite{TypeDKhov} in order to align with \cref{recap}, but the correspondence is often immediate. 
  
For each defect $k \in \mathbb{N}$, there exist exactly two principal blocks: one odd and one even, where the weight diagrams have no nodes labelled $\times$ or $\circ$. We denote the algebras associated to these blocks by $\mathbb{D}_{2k+1}$ and $\mathbb{D}_{2k}$, respectively. Since the algebra $\mathbb{D}_{\La}$ depends (up to canonical isomorphism) only on the defect and parity of $\La$, we will henceforth restrict our attention to the algebras $\mathbb{D}_n$ where $n \in \{2k, 2k+1\}$.

We begin by first constructing the individual elements of $\mathbb{D}_n$.
\begin{defn}
Let $\mu,\la \in \mptn$. We define the {\sf cap diagram} of $\la$, denoted $\overline{\mu}$, to be the cup diagram, $\underline{\mu}$, reflected through the horizontal-axis.
We obtain an oriented cap diagram $\la \overline{\mu}$ by glueing $\overline{\mu}$ on top of the weight $\la$. The cap diagram $\la \overline{\mu}$ is {\sf oriented} if and only if $\underline{\mu}\la$.

We define the degree of the oriented cap diagram $\la \overline{\mu}$ to be the same as the degree of $\la \overline{\mu}$. Equivalently, this is the number of caps whose right vertex is labelled by $\down$. Pictorially, this looks like:
\[\deg(\underline{\mu}\lambda)=
		\sharp\left\{  \begin{minipage}{1.2cm}
			\begin{tikzpicture}[scale=0.55]
				\draw[thick](-1,-0.155) node {$\boldsymbol\up$};
				\draw[very thick](0,0.1) node {$\boldsymbol\down$};
				\draw[very thick]  (-1,-0) to [out=90,in=180] (-0.5,0.75) to [out=0,in=90] (0,0);
			\end{tikzpicture}
		\end{minipage}
			\right\}	+
\sharp \left\{	\begin{minipage}{1.2cm}
			\begin{tikzpicture}[scale=0.55]
				\draw[thick](-1,0.1) node {$\boldsymbol\down$};
				\draw[very thick](0,0.1) node {$\boldsymbol\down$};
				\draw[very thick]  (-1,-0) to [out=90,in=180] (-0.5,0.75) to [out=0,in=90] (0,0);
				\draw[very thick, fill=black](-0.5,0.75) circle (4pt);
			\end{tikzpicture}
		\end{minipage}
		\right\}.	 \]
\end{defn}

\begin{defn}
Let  $\la,\mu \in \mptn$. 
We define the \textsf{circle diagram} $b=\underline{\la}\overline{\mu}$ to be the union of the arcs contained in the cup diagram $\underline{\la}$ and those in the cap diagram $\overline{\mu}$. Diagrammatically, we picture this simply by placing $\underline{\la}$ and $\overline{\mu}$ on the same horizontal axis.
\end{defn}

\begin{defn}
Let $A=\overline{\alpha}\underline{\alpha}$ be a type $D$ Temperley-Lieb diagram (as defined in \cite{MR1618912}) that is symmetric across the horizontal axis.
An (admissible) \textsf{stacked circle diagram} of height one is a sequence, $d=\underline{\la}A\overline{\mu}$, of two circle diagrams of the form $\underline{\la}\overline{\alpha}, \underline{\alpha}\overline{\mu}$ such that $\la,\mu,\alpha\in \mptn$.

We represent a stacked circle diagram visually by placing the two circle diagrams on top of each other with $\underline{\la}\overline{\alpha}$ at the bottom and $\underline{\alpha}\overline{\mu}$ on the top. Note that the middle section consists of the symmetric type $D$ Temperley-Lieb diagram $A=\overline{\alpha} \underline{\alpha}$.

\end{defn}

    \begin{figure}[ht!]
     $$ 
\begin{minipage}{5.3cm}
 
$$ 
    
    \caption{Above are the two circle diagrams $\overline{\mu}\underline{\alpha}$ and $\overline{\alpha}\underline{\la}$ and below the stacked circle diagram $\underline{\la}A\overline{\mu}$, where $A=\underline{\alpha}\overline{\alpha}$ is a type $D$ Temperely-Lieb diagram symmetric across the horizontal axis.}
    \label{stacking}
    \end{figure}

\begin{rmk}
We only consider the number of decorations on a given arc in a cup diagram modulo two. For example, if $\underline{\alpha}$ contains a decorated ray, then in $A=\overline{\alpha}\underline{\alpha}$ this strand is decorated twice. In the stacked circle diagram $D=\underline{\la}A\overline{\mu}$, we therefore picture the strand as undecorated, just like the leftmost strand of \cref{stacking}. 
\end{rmk}

\begin{defn}
A \textsf{line} in a (stacked) circle diagram, $D$, is any connected component of $D$ that intersects the northern or southern boundary of $D$. A line is \textsf{propagating} if it intersects both boundaries. 
A \textsf{circle} in a (stacked) circle diagram, $D$, is a connected component of $D$ that does not intersect either boundary of $D$.
\end{defn}

When multiplying elements of $\mathbb{D}_n$, it will be useful to specify coordinates in a stacked circle diagram. To avoid confusion with our notation for the vertices of cups, we will use square brackets $[x,y]$ where $x\in\frac{1}{2}+\mathbb{Z}_{\geq0}$ gives the horizontal position and $y=\{0,1\}$ specifies whether the vertex is in the top or bottom cup/cap diagram.

\begin{defn}
\label{Khov basis}
Let  $\la,\mu,\nu\in \mptn$.
Denote by $$\mathbb{B}_{n}=\{\underline{\la}\nu\overline{\mu}\mid \underline{\la}\nu, \nu\overline{\mu} \text{ are oriented cup (resp. cap) diagrams}\}$$
the set of \textsf{oriented circle diagrams} on $k$ vertices.
We define the degree of $\underline{\la}\nu\overline{\mu}$ to be the sum of the degrees of $ \underline{\la}\nu$ and $\nu\overline{\mu}$.
Then $\mathbb{D}_{n}$ is the graded $\Bbbk$-vector space with basis $\mathbb{B}_{n}$, where $\Bbbk$ is any field.
\end{defn}

The rest of this section is dedicated to defining an algebra multiplication on $\mathbb{D}_n$.

\begin{defn}
\label{something about orientation}
An \textsf{oriented connected component} in a (stacked) circle diagram is a component $C$, such that each arc $\gamma_i \in C$ is oriented, in the sense (i) that every undecorated arc has one vertex with one weight pointing into the strand and one pointing out of the strand and (ii) every decorated arc has either both nodes pointing into the arc or both nodes pointing out of the arc (cf. \cref{oriented or not?})
\end{defn}

\begin{defn}
Let  $D=\underline{\la}A\overline{\mu}$ be a stacked circle diagram of height $1$. An orientation of $D$ is a tuple $\boldsymbol{\nu}=(\nu_0,\nu_1)$ of two weights in $\mptn$ such that 
$\underline{\la}\nu_1\overline{\alpha}$ and $\underline{\alpha}\nu_1\overline{\mu}$ are both oriented.

An \textsf{oriented stacked circle diagram} is a stacked circle diagram $\underline{\la}A\overline{\mu}$ together with an orientation $\boldsymbol{\nu}$, which we will denote 
$\underline{\la}(A,\boldsymbol{\nu})\overline{\mu}$ or equivalently
$\underline{\la}\nu_0\overline{\alpha}\underline{\alpha}\nu_1\overline{\mu}$.

\end{defn}

   \begin{rmk}
 \label{myexcuse}
The central focus of \cite[Section 5]{TypeDKhov} is the proof of the non-trivial fact that multiplication in $\mathbb{D}_{n}$ is associative. Therefore, in \cite{TypeDKhov}, multiplication is explicitly defined for determining the product of an arbitrary number of ($k$) elements of $\mathbb{D}_{n}$. Consequently, \cite{TypeDKhov} defines multiplication on circle diagrams of arbitrary height to handle products of $n$ elements simultaneously. 
\end{rmk}

In the following, we present criteria for checking that $D=\underline{\la}A\overline{\mu}$ is orientable.

\begin{lem}
The following is a set of necessary and sufficient conditions that the stacked circle diagram  $D=\underline{\la}A\overline{\mu}$ must satisfy in order to be orientable.
\begin{itemize}
\item Every circle must have an even number of decorations.
\item Every propagating line must have an even number of decorations.
\item Every non-propagating line must have an odd number of decorations.
\end{itemize}
\end{lem}

\begin{proof}
The conditions in \cref{something about orientation} mean one can interpret decorations on an arc as orientation-reversing points. The first statement then follows immediately. Next, we note that for any line, if the two labels on the two rays are the same (respectively different), there must be an even (respectively odd) number of decorations on the strands between them. The rules for the decoration of rays from \cref{drawcups} then mean that the final two statements also follow.
\end{proof}
 
 \subsection{The decorated generalised surgery procedure}
 
To extend $\mathbb{D}_{n}$ into a graded algebra, we need to equip $\mathbb{D}_{n}$ with a method of multiplying oriented circle diagrams. In this subsection, we will describe how to multiply together a pair of oriented circle diagrams. In \cite[Section 5]{TypeDKhov}, it is shown (with great care) that this multiplication is associative. 
We define
\begin{equation}
\label{multeq}
(\underline{\mu}\la\overline{\nu})(\underline{\alpha}\beta\overline{\gamma})= \left\{ \begin{array}{ll}\sum_{i}\zeta_i(\underline{\mu}\sigma_{i}\overline{\gamma}) & \mbox{ if $(\overline{\nu})^*=\underline{\nu}=\underline{\alpha}$} \\
0 & \mbox{ otherwise}  
 \end{array} \right.
\end{equation}
where the precise weights $\sigma_i$ and coefficients, $\zeta_i$, will be described below via the \textit{decorated generalised surgery procedure}.

Suppose we have two oriented circle diagrams $(\underline{\la}\nu\overline{\alpha})$ and $(\underline{\alpha}\gamma\overline{\mu})$ (so the product $(\underline{\la}\nu\overline{\alpha})(\underline{\alpha}\gamma\overline{\mu})$ from \eqref{multeq} is non-zero). By placing $(\underline{\alpha}\gamma\overline{\mu})$ on top of $(\underline{\la}\nu\overline{\alpha})$ we create an oriented stacked circle diagram, $D$. To perform surgery, select a cup $\ps\in \underline{\alpha}$ that is not (doubly) covered by any $\cp \in \underline{\nu}$ and delete $\ps$ along with its corresponding cap in $\overline{\alpha}$. Replace this cup/cap pair with a pair of undecorated vertical lines between the two weights at $l_\ps$ and $r_\ps$ (always undecorated even if $\ps$ was decorated), hence obtaining a new stacked circle diagram, $D'=\underline{\la}B\overline{\mu}$ say (see the first equality of \cref{KhovMultEg0}). 
We will define the $\ps$-surgery of $D$  to be equal to a signed sum 
 over all possible orientations of $D'$, where the coefficients ($0$, $1$, or $-1$) are determined below. 
 
\begin{rmk}
\label{admissible}
We say that the diagram $D'=\underline{\la}B\overline{\mu}$ is {\sf admissible} if $B=\overline{\beta}\underline{\beta}$ is a horizontally symmetric Temperley-Lieb diagram. In particular, every decorated arc/strand must be able to be deformed isotopically to the left of the diagram. 
Suppose that performing a $\ps$-surgery to $D$ results in a circle diagram $D'$ that is not admissible. If there exists a cup $\ps\neq\cp \in \underline{\alpha}$ such that we can perform $\cp$-surgery to  $D$ and obtain an admissible diagram, we perform the $\cp$-surgery first. In practice, this means first performing surgery on the right-hand side 
 of any decorated $\ps \in \underline{\alpha}$. 
The need for this constraint can be seen in \cite[Example 6.7]{TypeDKhov}.
\end{rmk}

We can iterate the (admissible) surgery procedure until 
  the resulting stacked circle diagram we obtain is $\tilde{D}=\underline{\la}I\overline{\mu}$, where $I$ is the identity Temperley-Lieb diagram, containing only vertical lines in the centre. We delete these vertical lines and hence obtain a sum of oriented circle diagrams $B\in \mathbb{B}_{n}$, all of which have the same underlying 
  circle diagram but differ from one another in their orientations. An example of this process can be seen in \cref{KhovMultEg0}.

 \begin{figure}[ht!]
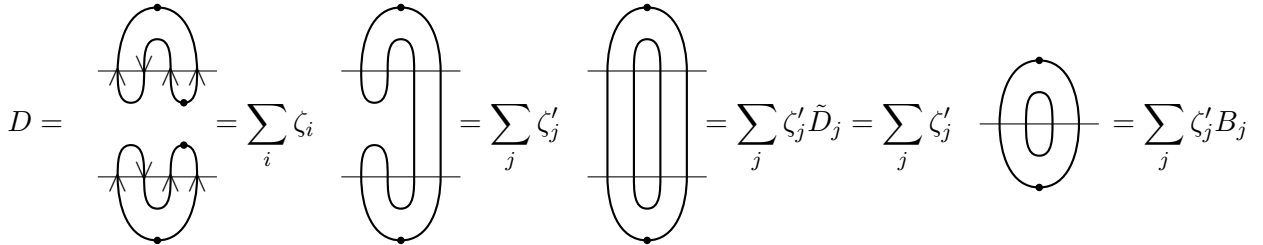

 $$D=\begin{minipage}{1.8cm}
 
\end{minipage}=\sum_{j} \zeta_j'B_j$$
	\caption{An example of the surgery process in action, where the sums are taken over weights that result in an orientable circle diagram and the coefficients $\zeta_i$ are determined by the surgery procedure.
	By  \cref{admissible}  our choice of  the first surgery to be performed at $(l_\ps,r_\ps)=(\frac{5}{2},\frac{7}{2})$ was the only possible choice (first 
	performing $\cp$-surgery for  	$\cp=(\frac{1}{2},\frac{3}{2})$ would result in an inadmissible diagram).}
	\label{KhovMultEg0}
\end{figure}

\begin{rmk}
\label{only two}
To construct an explicit isomorphism with $\mathcal{H}_{(D_n, A_{n-1})}$, which is similarly presented by generators in degrees 0 and 1 with quadratic relations (\cref{presentation}), it will suffice to map generators to generators and verify that the degree two relations hold in $\mathbb{D}_{n}$. Accordingly, we restrict our definitions to the multiplication of pairs of elements. While this renders our exposition less general than that of \cite{TypeDKhov}, the extension to the general case is natural.
\end{rmk}

The coefficients, $\zeta_i$, and orientations, $\sigma_i$, obtained during the surgery process of \cref{multeq} are determined by the number of connected components in the stacked circle diagrams $D$ and $D'$. There are three possible situations:
\begin{itemize}
\item A \textbf{Merge} is when the surgery reduces the number of connected components (by one).
\item A \textbf{Split} is when the surgery increases the number of connected components (by one).
\item A \textbf{Reconnect} is when the surgery keeps the number of connected components fixed.
\end{itemize}
We will give the exact formulae for each of these situations shortly, but first, we need to define some features of oriented stacked circle diagrams.

\begin{defn}
\label{tag}
Suppose $C$ is a connected component inside an oriented stacked circle diagram $D$. 
We define the \textsf{tag} of $C$, $t(C)$, to be the rightmost vertex which $C$  intersects (if there are two choices for $t(C)$, we choose the northernmost).  
The orientation of $C$ is determined by the symbol at $t(C)$:
\begin{itemize}
\item A circle $C$ is oriented \textsf{clockwise} if the symbol at $t(C)$ is $\down$.
\item A circle $C$ is oriented \textsf{anti-clockwise} if the symbol at $t(C)$ is $\up$.
\item Any line is always said to be  \textsf{anti-clockwise} oriented.
\end{itemize}
We write ${\sf orient}(C)=\down$ if $C$ is a clockwise component (and similarly ${\sf orient}(C)=\up$ for anti-clockwise components).   
 \end{defn}

We next define a constant associated with the connected components of a diagram that will be needed to calculate the coefficients $\zeta_i$.

\begin{defn}
\label{signs up}
Suppose $C$ is a connected component inside a (oriented) stacked circle diagram $D$ containing the vertex $[i,y]$, for $i\in\frac{1}{2}+\mathbb{Z}_{\geq0}$ and $y\in\{0,1\}$. 
We define $$
{\sf sign}_D[i,y]= 
\begin{cases}
1		&\text{if $ {\sf orient}(C)$ matches the symbol at $[i,y]$}\\
-1						&\text{otherwise.}
\end{cases}
$$
See \cref{BigHopefullyUseful} for an example.
\end{defn}

\begin{rmk}
We should note that this differs from the definition used in \cite{TypeDKhov}, but the equivalence of the two definitions can be easily seen.
\end{rmk}

\begin{figure}
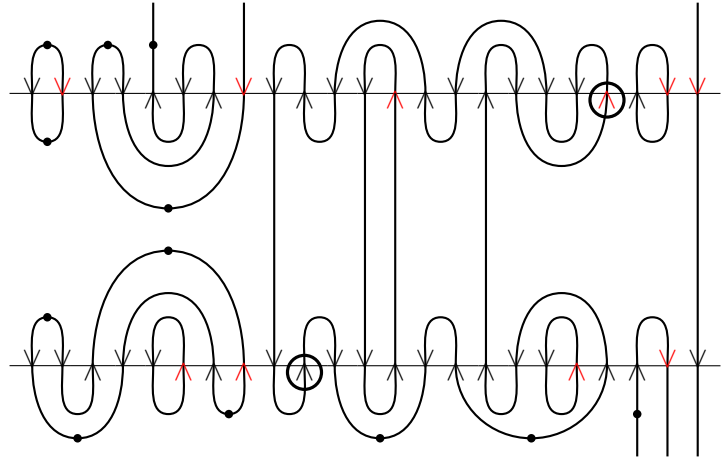

 
\caption{An oriented stacked circle diagram, $D$, where the tag for each connected component has been highlighted in red. We have that ${\sf sign}_D[\frac{19}{2},0]=1$ with the label at $[\frac{19}{2},0]$ and the corresponding tag of the component circled.}
	\label{BigHopefullyUseful}
\end{figure}

We are now ready to give explicit rules for the decorated generalised surgery procedure. Throughout we suppose that performing an $\ps$-surgery to an oriented stacked circle diagram, $D=\underline{\la}(A,\boldsymbol{\nu})\overline{\mu}$, creates the oriented stacked circle diagram $D'=\underline{\la}B\overline{\mu}$:

\subsubsection{Merge}
This is when an $\ps$-surgery replaces two connected components with a single component (for $\ps=(l_\ps , r_\ps)$). 
Let  $C_1$ denote the component of 
$D=\underline{\la}(A,\boldsymbol{\nu})\overline{\mu}$ containing 
$[l_\ps, 1]$;
let $C_0$ denote  the component of $D$ containing $[l_\ps, 0]$;
 and let $C$ denote the component in $D'=\underline{\la}B\overline{\mu}$ containing both $[l_\ps, 0]$  and $[l_\ps, 1]$.
\begin{equation}
\label{mergeeq}
\underline{\la}(A,\boldsymbol{\nu})\overline{\mu}= \left\{ 
\begin{array}{ll} \underline{\la}(B,\boldsymbol{\nu'})\overline{\mu} & \mbox{ if 
${\sf orient}(C_0)=\up = {\sf orient}(C_1)$,} \\
 ({\sf sign}_D[l_\ps,0]{\sf sign}_{D'}[l_\ps,0] ) \underline{\la}(B,\boldsymbol{\nu''})\overline{\mu} & \mbox{ if 
${\sf orient}(C_0)=\down$,   $ {\sf orient}(C_1)= \up$,} \\
 ({\sf sign}_D[l_\ps,1]{\sf sign}_{D'}[l_\ps,1])  \underline{\la}(B,\boldsymbol{\nu''})\overline{\mu} & \mbox{ if
${\sf orient}(C_0)=\up$,   $ {\sf orient}(C_1)= \down$,} \\
0 & \mbox{ if ${\sf orient}(C_0)=\down= {\sf orient}(C_1) $,}  
 \end{array} \right.
\end{equation}
where $\boldsymbol{\nu}'$ is obtained by changing $\boldsymbol{\nu}$ 
only to orient $C$ anti-clockwise;
 and $\boldsymbol{\nu''}$ is obtained by orienting $C$ clockwise.
If $C$ is a line and so cannot be oriented clockwise, then we define $\underline{\la}(B,\boldsymbol{\nu''})\overline{\mu}$ to be equal to $0$.
For two simple examples, see \cref{In turn}.

\begin{figure}[ht!]
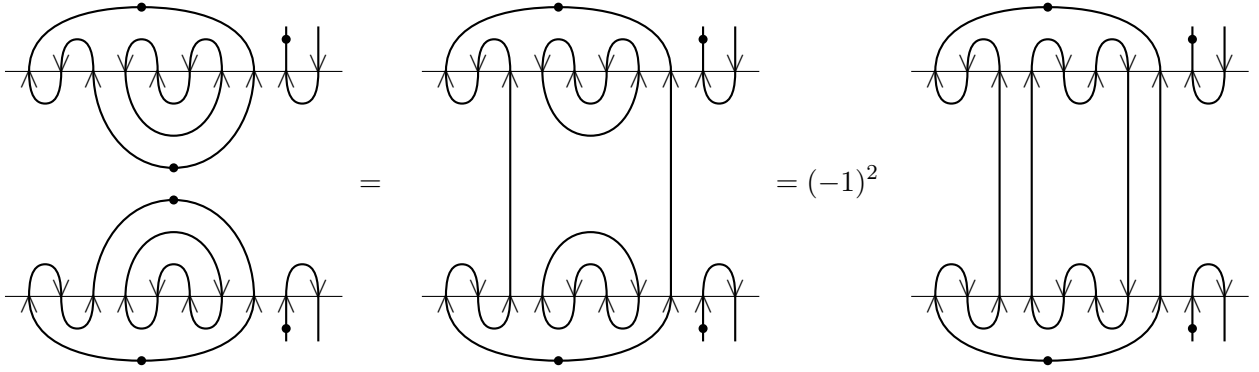

$$\begin{minipage}{5cm}%
 
\end{minipage}
$$
 \caption{In the first equality we perform $\ps$-surgery for $\ps=(\frac{5}{2}, \frac{15}{2})$ and in the second equality we perform $\cp$-surgery for $\cp=(\frac{7}{2}, \frac{13}{2})$. Note that in the second surgery $C_0$ is oriented clockwise and hence the merged circle, $C$, is oriented clockwise and we pick up a factor of $(-1)^2$, where $-1={\sf sign}_D[\frac{7}{2},0]={\sf sign}_{D'}[\frac{7}{2},0]$}
 \label{In turn}
\end{figure}

\subsubsection{Split}
This is when an $
\ps$-surgery replaces one connected component with two connected components (for $\ps=(l_\ps, r_\ps)$).   We let $C$ denote 
 the connected component of $D=\underline{\la}(A,\boldsymbol{\nu})\overline{\mu}$ 
 containing  
 $[l_\ps,1]$ and $[r_\ps,1]$;
  we let $C_l$ denote the component of $D'=\underline{\la}B\overline{\mu}$ containing $[l_\ps, 1]$;
  and we let $C_r$ the component in $D'$ containing $[r_\ps, 1]$. 
\begin{equation}
\label{spliteq}
\underline{\la}(A,\boldsymbol{\nu})\overline{\mu}
= \left\{ 
\begin{array}{ll}
(-1)^{\left \lfloor{l_\ps}\right \rfloor }\Bigl({\sf sign}_{D'}[r_\ps,1]  \underline{\la}(B\boldsymbol{\nu'})\overline{\mu}- {\sf sign}_{D'}[l_\ps,1]  \underline{\la}(B,\boldsymbol{\nu''})\overline{\mu}\Bigr) &\mbox{\begin{minipage}{3.5cm} 
if ${\sf orient}(C )=\up$ and $\ps$ is undecorated,
\end{minipage}} \\[10pt]
(-1)^{\left \lfloor{l_\ps}\right \rfloor }\Bigl({\sf sign}_{D'}[r_\ps,1]  \underline{\la}(B\boldsymbol{\nu'})\overline{\mu}+{\sf sign}_{D'}[l_\ps,1] \underline{\la}(B,\boldsymbol{\nu''})\overline{\mu}\Bigr)&\mbox{\begin{minipage}{3.5cm}   if 
${\sf orient}(C )=\up$ and $\ps$ is decorated,
\end{minipage}
}  \\[10pt]
(-1)^{\left \lfloor{l_\ps}\right \rfloor }\Bigl({\sf sign}_D[l_\ps,1]{\sf sign}_{D'}[l_\ps,1]{\sf sign}_{D'}[r_\ps,1]\Bigr)  \underline{\la}(B,\boldsymbol{\nu'''})\overline{\mu} &\mbox{if ${\sf orient}(C )=\down$,} \\
0 &\mbox{if $D'$ is not orientable,}  
 \end{array} \right.
\end{equation}
where $\boldsymbol{\nu'''}$ is obtained by changing $\nu$ such that $C_r$ and $C_l$ are oriented clockwise. $\boldsymbol{\nu'}$ is obtained by orienting $C_l$ is anti-clockwise and $C_r$ clockwise and  $\boldsymbol{\nu''}$ is obtained by orienting $C_l$ is clockwise and $C_r$ anti-clockwise. See \cref{Splitting} for an example.

\begin{figure}[ht!]
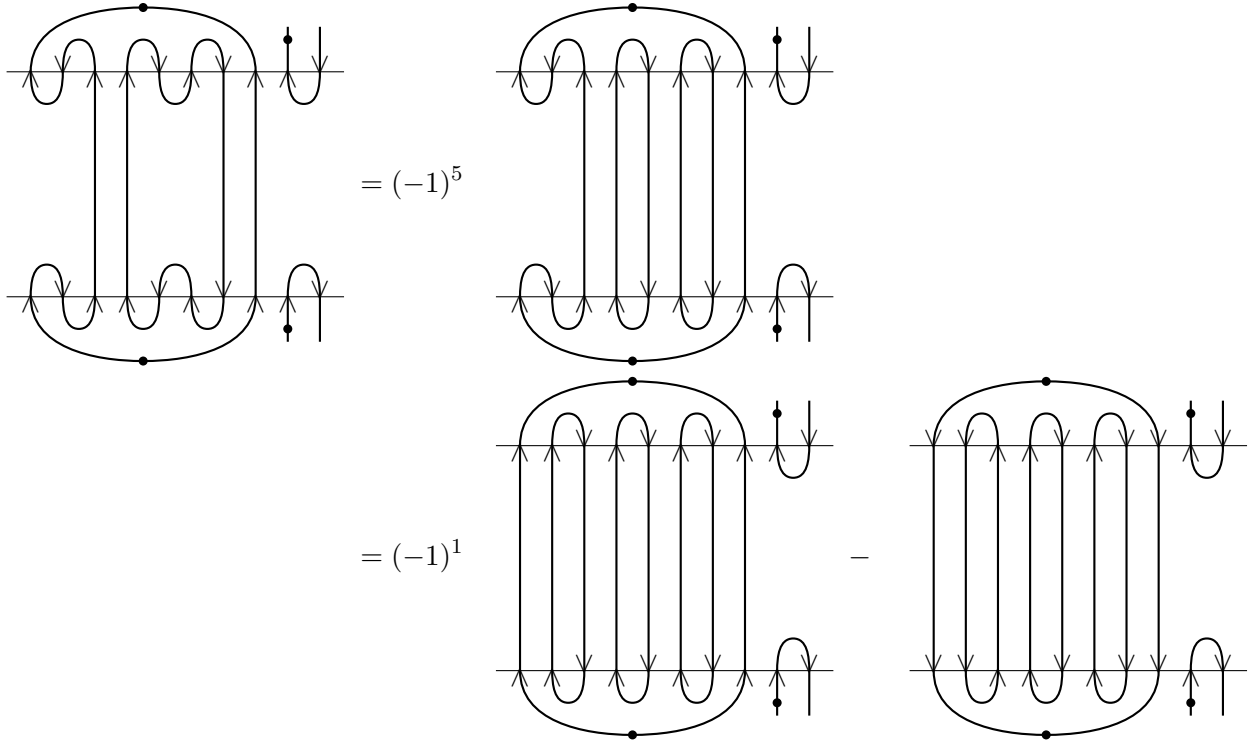

\begin{equation*}
\begin{aligned}
\begin{minipage}{5cm}%
 
\end{minipage}
\end{aligned}
\end{equation*}
 \caption{In the first equality we perform $\mq$-surgery for $\mq=(\frac{9}{2}, \frac{11}{2})$ and in the second equality we perform $\ot$-surgery for $\ot=(\frac{1}{2}, \frac{3}{2})$. In the first equality, $C$ is oriented clockwise, and hence both circles are oriented clockwise. We pick up a factor of $(-1)^5$, from the fact that $-1={\sf sign}_{D'}[\frac{11}{2},,1]$ and $1={\sf sign}_{D}[\frac{9}{2},1]={\sf sign}_{D'}[\frac{9}{2},1]$. In the $\ot$-surgery, $C$ is oriented anti-clockwise, and so we get a sum of two oriented circle diagrams where either $C_r$ or $C_l$ is oriented clockwise. Noting that $-1={\sf sign}_{D'}[\frac{3}{2},1]$, $1={\sf sign}_{D'}[\frac{1}{2},1]$ and $\ot$ is undecorated gives us the required signs.}
 \label{Splitting}
\end{figure}

\subsubsection{Reconnect}
This is when $\ps$-surgery preserves the number of connected components. This only happens if the cup and cap lie on two distinct (propagating or non-propagating) lines in $D=\underline{\la}(A,\boldsymbol{\nu})\overline{\mu}$B.
\begin{equation}
\label{reconeq}
\underline{\la}(A,\boldsymbol{\nu})\overline{\mu}= \left\{ 
\begin{array}{ll}\underline{\la}(B,\boldsymbol{\nu})\overline{\mu}& \mbox{if $D'=\underline{\la}B\overline{\mu}$ is orientable and the two lines in $D$ are propagating,} \\
0 & \mbox{otherwise,}  
 \end{array} \right.
\end{equation}
See \cref{Zero} for an example of a reconnect surgery.

\begin{figure}[ht!]
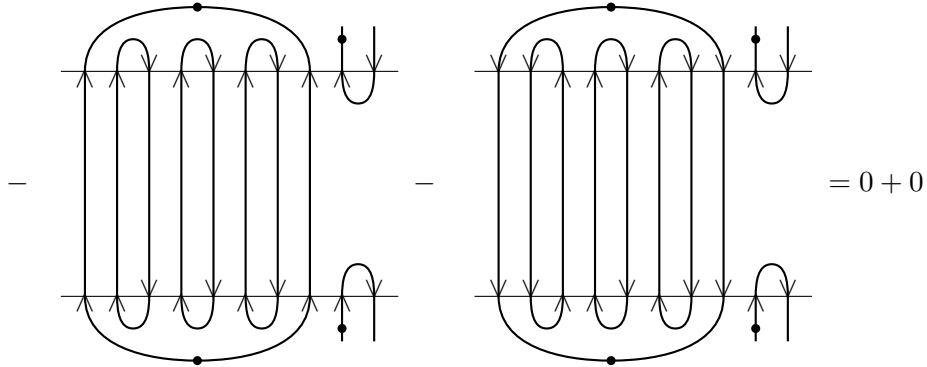

$$
-\begin{minipage}{5cm}%
 
\end{minipage}=0+0
$$
 \caption{For both terms of the sum, $\gr$-surgery at $\gr=(\frac{17}{2}, \frac{19}{2})$ is zero as both lines involved are non-propagating.}
 \label{Zero}
\end{figure}

\begin{defn}
We define the generalised arc algebra $\mathbb{D}_n$ to be the graded $\Bbbk$-vector space with basis $\mathbb{B}_{n}$, where $\Bbbk$ is any field, equipped with the multiplication from \cref{multeq}, as outlined in the previous subsection.
\end{defn}

\begin{rmk}
\label{get out of jail remark}
We note that the presentation of the algebra differs only slightly from the definition given in \cite{TypeDKhov}, specifically by multiplying the left-hand side of
 \eqref{spliteq} by $-1$.  
 This is due to a slight error in that paper (which results in their algebra failing to be associative). 
 We will discuss this in more detail (and demonstrate the failure of associativity in a small example) in the appendix to this paper. We thank Eliot Grimont for pointing this out to us.
%
 \end{rmk}

\color{red}

%
\color{black}
One can find further examples of this multiplication in \cite[Section 6.3]{TypeDKhov} (noting the slight change highlighted in \cref{get out of jail remark}), where the importance of performing surgery to maintain admissible circle diagrams is also highlighted.

The following result tells us that $\mathbb{D}_n$ is generated by precisely its degree $0$ and $1$ elements.

\begin{thm}[\cite{TypeDKhov}, Theorem 6.10]
\label{generated by in Khov}
The algebra $\mathbb{D}_n$ is generated by $\{\underline{\la}\la\overline{\la}\}\cup\{\underline{\la}\la\overline{\mu}\mid \la=\mu-\cp, \cp \in \underline{\mu}\}$
\end{thm}

 \section{Contraction}
 \label{KhovContract}

In this section we construct contraction maps. These will allow us to compare weights, cups, Hecke categories and generalised arc algebras of different sizes. In the proof of Theorem A, this will allow us to limit our sights to \lq \lq small cases", where the cup diagrams and Khovanov arc diagrams have the minimal number of vertices. This will justify our choice to picture the adjacent cases from \cref{adj lem0,adj lem1,adj lem2} on a minimal number of vertices. The contraction map will provide a quick and easy way to generalise smaller cases to all of $\mathcal{H}_{(D_n, A_{n-1})}$ and $\mathbb{D}_n$. As such, this will significantly reduce the intricacy of the calculations involved.

\begin{defn}
For $k\in \mathbb{Z}_{>0}$, define $\mathscr{P}^k_{n}$ to be the cup diagrams with a cup $\cp \in \underline{\mu}$ such that $(l_\cp,r_\cp)=(k-\frac{1}{2},k+\frac{1}{2})$.
Similarly $\mathscr{P}^0_{n}$, if and only if there exists a decorated $\mq \in \underline{\la}$ with $(l_\mq,r_\mq)=(\frac{1}{2},\frac{3}{2})$
We call such weights {\sf contractible at $k$}.
 \end{defn}

\begin{defn}
For  $0 \leq k \leq n$ we define the {\sf contraction map} $\Phi_k:  \mathscr{P}^k_{n} 
 \longrightarrow
  \mathscr{P}_{n-2}$ on weights by setting $\Phi_k (\la)$ for $\la \in  \mathscr{P}^k_{n}$ to be the weight obtained from $\la$ by removing these vertices.
  \end{defn}
 
\begin{rmk}

Strictly speaking, we $(i)$ fix the vertices at $x=\frac{1}{2}, \dots, k-\frac{3}{2}$, $(ii)$ then delete the vertices at $x=k-\frac{1}{2},k+\frac{1}{2}$, $(iii)$ then shift the vertices at $x=k+\frac{3}{2},\dots,n-\frac{1}{2}$ to $x-2$ and $(iv)$ finally  orientate the first node (at $x=\frac{1}{2}$) as required to maintain an even number of $\up$ vertices.
See \cref{dilation on cups} for some examples of this.
\end{rmk}

    \begin{figure}[ht!]
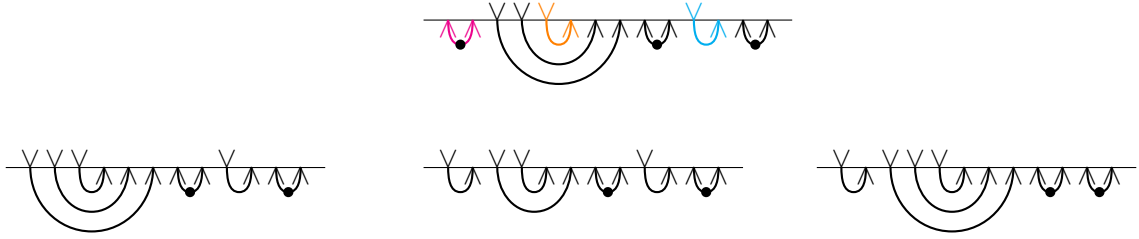

     $$   
$$ 
    
    \caption{$\la=(1,2,3,4,5,6,7,8,9,8,8,8,3)$ is contractible at $k=0,15,11$. We depict below the contraction of $\la$ at $k={\color{magenta}0},{\color{orange}5}$ and ${\color{cyan}11}$ respectively; these are all $k$ for which $\la$ is contractible.}
    \label{dilation on cups}
    \end{figure} 
 
The following lemmas follow directly from the definitions.

\begin{lem}
The map $\Phi_k$ is bijective.
\end{lem}

\begin{lem} 
\label{degree}
 Let $\la, \mu \in  \mathscr{P}^k_{n}$.
We have that $\DP$ is oriented 
if and only if $\underline{\Phi_k(\mu)}\Phi_k(\la)$ is oriented.
If $\la= \mu - \sum_{i=1}^{j}\mq^i$ and for some $1\leq i \leq l$, we have that $(l_{\mq^i},r_{\mq^i})= (k-\frac{1}{2},k+\frac{1}{2})$, then $\DP$ is of degree j and $\underline{\Phi_k(\mu)}\Phi_k(\la)$ is of degree $j-1$. If no such $\mq^i$ exists then both $\DP$ and $\underline{\Phi_k(\mu)}\Phi_k(\la)$ are both of degree j. 
The equivalent statement for the oriented cap diagram $\la\overline{\mu}$ follows likewise.
\end{lem}

\begin{lem}
\label{breadth}
If $\la = \mu - \cp$ for some $\cp\in \underline{\mu}$ with $(l_{\cp},r_{\cp})\neq (k-\frac{1}{2},k+\frac{1}{2})$  then we have $\Phi_k(\la) = \Phi_k(\mu) - \cp'$ where $\cp'\in \underline{\Phi_k(\mu)}$   satisfies 
\begin{center}
\begin{tabular}{ll}
 $b(\cp')=b(\cp)-2$ & if $k < l_{\cp}$ and $\cp$ is decorated, \\
$b(\cp')=b(\cp)-1$ & if $l_{\cp}<k <r_{\cp}$,\\
 $b(\cp')=b(\cp)$ & if $k> r_{\cp}$ or $k < l_{\cp}$ and $\cp$ is undecorated.
\end{tabular}
\end{center}
Furthermore, the decorations of 
  $\cp \in  \mu $  and 
  $\cp'\in \underline{\Phi_k(\mu)}$  
are identical unless      $l_{\cp'} = \frac{1}{2}$ and $k\neq0$. 

In all cases we write $\Phi_k(\cp):=\cp'$.
\end{lem}

 We now extend the contraction map $\Phi_k$ to homomorphisms for the Hecke category. We will use the same notation for both contraction maps. 

\begin{thm}\cite[Theorem 6.8]{My1}
\label{dilation}
Let $\Bbbk$ be a commutative integral domain and let $i \in \Bbbk$ be a square root of $-1$ and set $1_k =  \sum_{\mu \in  \mathscr{P}^k_{n}} 1_\mu$.
We define the maps $$\Phi_k : 1_k\mathcal{H}_{(D_n, A_{n-1})}1_k \to \mathcal{H}_{ (D_{n-2}, A_{n-3}) }$$on the generators as follows.
Suppose $\la, \mu \in \mathscr{P}^k_{n} $, with $\la = \mu - \cp$, for some $\cp \in \underline{\mu}$ with $(l_{\cp},r_{\cp})\neq (k-\frac{1}{2},k+\frac{1}{2})$ , we define $\Phi_k({\sf 1}_\mu)= {\sf 1}_{{\Phi_k(\mu)}}$
and 
$$
\Phi_k(D^\la_\mu)= \left\{ \begin{array}{ll} -D^{\Phi_k(\la)}_{\Phi_k(\mu)} & \mbox{ if $\cp$ is decorated, $k=0,1$ and $k<l_{\cp}$} \\
i \cdot 
D^{\Phi_k(\la)}_{\Phi_k(\mu)}	& \mbox{ if $l_{\cp}<k <r_{\cp}$ and $\left\lceil l_p\right\rceil-k \equiv 0\pmod{2}$}  
\\	
(-i) \cdot 
D^{\Phi_k(\la)}_{\Phi_k(\mu)}	& \mbox{ if $l_{\cp}<k < r_{\cp}$ and $\left\lceil l_p\right\rceil-k \equiv 1\pmod{2}$}\\
D^{\Phi_k(\la)}_{\Phi_k(\mu)} & \mbox{ else.}
 \end{array} \right.
$$Furthermore, we define $\Phi_k(D_\la^\mu) = \Phi_k((D_\mu^\la)^*) = (\Phi_k(D_\mu^\la))^*$. Then $\Phi_k$ extends to an isomorphism of graded $\Bbbk$-algebras. 
\end{thm}

\begin{defn}
Let $\la, \mu \in \mptn$ with $\la = \mu-\cp$ for some $\cp \in \underline{\mu}$. We say that $\underline{\mu}\la$ is \textsf{incontractible} if there doesn't exist $k\in \ZZ_{\geq0}$ such that $\la, \mu \in\mathscr{P}^k_{n}$.
\end{defn}

\begin{rmk}
It follows from the definition that $\underline{\mu}\la$ being incontractible is equivalent to $\cp$ being the unique cup in $\underline{\mu}$ with $b(\cp)=1$.
\end{rmk}

Finally, we extend the contraction map to the generalised arc algebra as well. We begin with a small lemma that we use immediately.
 
\begin{lem}[A local idempotent]
\label{idempot lem}
Any anti-clockwise oriented circle which intersects the weight at exactly two points acts as a local idempotent in the following sense: Applying a surgery to this circle is equivalent to deleting the circle (see \cref{local idempot} for an example of this in action). We call a circle that intersects the weight at exactly two points \textsf{small}.
\end{lem}

\begin{figure}[ht!]
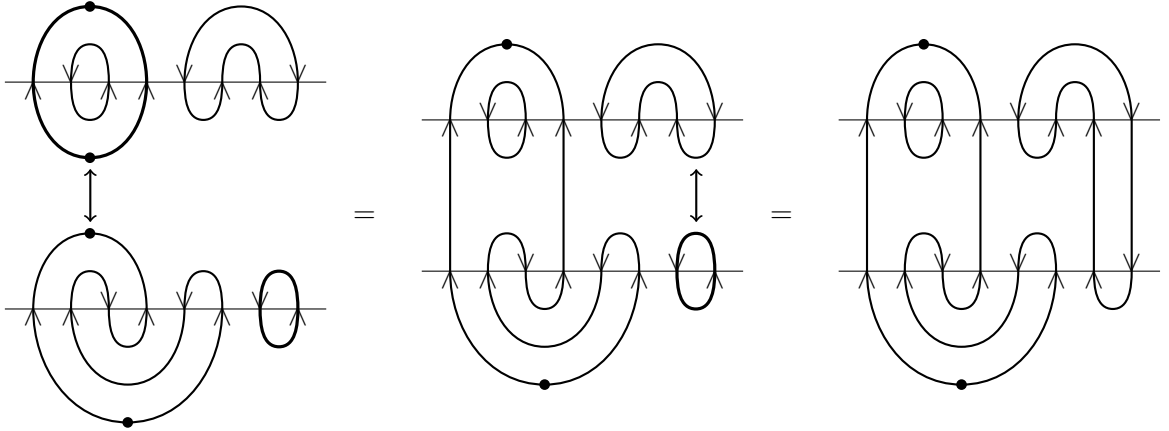

$$\begin{minipage}{5cm}
 
\end{minipage}
$$
 \caption{An example of the local idempotent multiplication. Note that the surgery performed can be thought of as simply deleting the highlighted anti-clockwise circles (with their weights) and `stretching' the opposite component across the centre of the diagram.}
 \label{local idempot}
\end{figure}

\begin{proof}
By assumption, the circle only intersects weights at a given fixed level (either $0$ or $1$); therefore, any surgery performed at this point will be a merge. Applying \cref{mergeeq}, the orientation of the merged circle is clockwise if and only if the non-small circle was oriented clockwise. Since the circle only intersects the weights twice, the choice of tag for the other component can remain constant, so any signs cancel out.
\end{proof}

 \begin{prop}
 \label{KContraction homomorphism}
 Let $\Bbbk$ be a commutative integral domain. Denote by $\mathbb{D}_{n}^k$ the subalgebra of $\mathbb{D}_{n}$ containing only the diagrams $\underline{\la}\nu\overline{\mu}$ such that $\la, \mu, \nu \in \mathscr{P}^k_{n}$. Then the map $\tilde\Phi_k:\mathbb{D}_{n}^k \rightarrow \mathbb{D}_{n-2}$ defined on generalised arc diagrams by 
 $$\tilde\Phi_k(\underline{\la}\nu\overline{\mu})=\Phi_k(\underline{\la})\Phi_k(\nu)\Phi_k(\overline{\mu})
 $$
 is an isomorphism of graded $\Bbbk$-algebras.
   \end{prop}
   
 \begin{proof}
 Our assumption that $\la, \mu, \nu \in \mathscr{P}^k_{n}$ implies that the diagram $\underline{\la}\nu\overline{\mu}$ has an anti-clockwise circle that intersects weights at only the points $k, k+1$. When multiplying diagrams, we can choose to perform the surgery procedure, so as to leave this circle until the final step. This final surgery acts as a local idempotent as in \cref{idempot lem}. The result then follows.
  \end{proof}

  \section{Isomorphism between the Hecke category and Khovanov arc algebras}
  \label{TheActualIso}

  We now utilise our combinatorial presentation of $\mathcal{H}_{(D_n, A_{n-1})}$ to prove Theorem A. That is, we construct a graded $\mathbb{Z}[q,q^{-1}]$-algebra isomorphism to $\mathbb{D}_n$, over any field that contains a square root of $1$.
    
\subsection{The explicit isomorphism}

As we will see, the isomorphism in \cref{THE isomorphism} is remarkably simple: it maps the degree zero and one generators of $\mathcal{H}_{(D_n, A_{n-1})}$ to the corresponding generators of $\mathbb{D}_n$, up to a constant. We begin this section by defining this constant.

\begin{defn}\label{suppminusdefn}
Let $\la, \mu \in \mptn$ be such that $\la = \mu-\cp$. Define $\ka(\la,\mu)\coloneqq \frac{1}{2}(l_\cp +r_\cp)$. See \cref{sgn examples} for a simple example.
\end{defn}

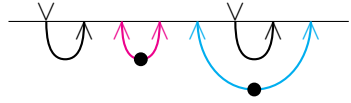
\begin{figure}[ht!]
     $$   \begin{tikzpicture} [scale=1]

		   \path (0,6.5) coordinate (origin); 
		\path (origin)--++(0.5,0.5) coordinate (origin2);  
	 	\draw(origin2)--++(0:4.5); 
		\foreach \i in {1,2,3,4,5,...,8}
		{
			\path (origin2)--++(0:0.5*\i)--++(-90:0.00) coordinate (c\i); 
			  }
		
		\foreach \i in {1,2,3,4,5,...,16}
		{
			\path (origin2)--++(0:0.25*\i) --++(-90:0.5) coordinate (b\i); 
			\path (origin2)--++(0:0.25*\i) --++(-90:0.9) coordinate (d\i); 
		}
		\path(c4) --++(-90:0.13) node  {  \color{magenta} $  \up   $} ;
		\path(c2) --++(90:-0.13) node  {  \color{black} $  \up   $} ;
		\path(c3) --++(90:-0.13) node  { \color{magenta}$  \up   $} ;
		\path(c1) --++(90:0.13) node  {  \color{black}$  \down   $} ;
		\path(c6) --++(90:0.13) node  { \color{black}$  \down   $} ;
		\path(c5) --++(-90:0.13) node  {  \color{cyan}$  \up   $} ;
		\path(c7) --++(-90:0.13) node  {  \color{black}$  \up   $} ;
		\path(c8) --++(-90:0.13) node  { \color{cyan}$  \up   $} ;

		\draw[ thick, cyan](c8) to [out=-90,in=0] (d13) to [out=180,in=-90] (c5); 
		\draw[ thick, black](c7) to [out=-90,in=0] (b13) to [out=180,in=-90] (c6); 
		\draw[ thick, magenta](c4) to [out=-90,in=0] (b7) to [out=180,in=-90] (c3); 
		\draw[   black, thick](c2) to [out=-90,in=0] (b3) to [out=180,in=-90] (c1);

		\draw[very thick,  fill=black](b7) circle (2pt);
		\draw[very thick,  fill=black](d13) circle (2pt);

		\end{tikzpicture}$$

    \caption{In the cup diagram, $\underline{\mu}$, pictured the cups we have the that $\ka(\mu-\mq,\mu)=3$ and $\ka(\mu-\cp,\mu)=6$.}
    \label{sgn examples}
    \end{figure}

\begin{thm}
\label{THE isomorphism}
\color{black}
Let $\Bbbk$ be an integral domain containing $i \in \Bbbk$ such that $i^2 = -1$. 
We define a map $\Psi: \mathcal{H}_{(D_n, A_{n-1})} \rightarrow \mathbb{D}_{n}$ on generators as follows: if $\la = \mu - \cp$, we set$$\Psi(1_\la) = \underline{\la}\la\overline{\la}, \qquad \Psi(D_\mu^{\la}) = i^{\ka(\la,\mu)}\underline{\la}\la\overline{\mu}, \qquad \Psi(D_\la^{\mu}) = i^{\ka(\la,\mu)}\underline{\mu}\la\overline{\la},$$and extend $\Psi$ linearly. 

Then $\Psi$ is a $\mathbb{Z}$-graded isomorphism of $\Bbbk$-algebras.
\end{thm}

\begin{rmk}
Before proceeding to the proof of \cref{THE isomorphism}, we'll take a very brief pause to talk about the role played by the constant $\ka(\la,\mu)$ above. 
When comparing the results of \cref{dilation} and \cref{KContraction homomorphism}, one will notice that the contraction map from $\mathcal{H}_{(D_n,A_{n-1})}$ involves powers of $i$ whereas the equivalent map from $\mathbb{D}_n$ does not. 
This is what the constant $\ka(\la,\mu)$ accounts for; it is some measure of how close to being incontractible (or not) degree one elements are.

Given the complicated nature of the signs in the surgery procedure of (and therefore multiplication inside) $\mathbb{D}_n$, it is perhaps reassuring to see this information restricted to a single constant when considering the map $\Psi$.
\end{rmk}

\subsection{Proof of the isomorphism}
The remainder of this paper is devoted to the proof of \cref{THE isomorphism}.

The map $\Psi$ maps the complete set of degree $0$ and $1$ elements of $\mathcal{H}_{(D_n,A_{n-1})}$ to the complete set of degree $0$ and $1$ elements in $\mathbb{D}_n$. We will first prove that $\Psi$ is a $\mathbb{Z}$-graded homomorphism of $\Bbbk$-algebras, by checking that the relations listed in \cref{presentation} also hold in the image of $\Psi$.

For the sake of page management, we have chosen to picture all diagrams in what follows, as if they were in the smallest $\mathbb{D}_{n}$ possible. The natural embedding  $\mathbb{D}_{n} \rightarrow \mathbb{D}_{n+m}$ defined by adding undecorated rays to the right-hand side (or equivalently the embedding $\mathscr{P}_{n} \rightarrow \mathscr{P}_{n+m}$ defined by adding $\down$ nodes to the right), which play no role in the multiplication, means that we lose nothing by drawing the diagrams in the smallest $\mathbb{D}_{n}$ possible.

Due to \cref{KContraction homomorphism}, it will suffice in most of the proof of \cref{THE isomorphism} to only perform calculations on incontractible cup diagrams. In the following, we make explicit the possible shapes of these.
\begin{rmk}[Incontractible cup diagrams]
\label{Incontractible cup diagrams}
Suppose $\la=\mu-\cp$. 
By definition, $\underline{\mu}\la$ being incontractible is equivalent to $\cp$ being the unique cup in $\underline{\mu}$ with $b(\cp)=1$
 Hence, there are three general cases of incontractible diagrams, which we detail below.
\begin{itemize}
\item If $\cp$ is covered, then $\mu=(1,2,\dots,c^k)$ and $\cp$ is undecorated, we are in one of the following cases;
\begin{itemize} \item 
If $c\geq k$ and $c$ is even, then $\cp$ is covered by $k-1$ undecorated cups and doubly covered by $\lfloor \frac{c-k}{2}\rfloor$ decorated cups, as pictured on the left of \cref{Inconcup1}. If $c\geq k$ and $c$ is odd, then $\cp$ is covered by one decorated cup and $k-2$ undecorated cups, and also doubly covered by $\lfloor \frac{c-k}{2}\rfloor$ decorated cups as pictured on the right of \cref{Inconcup1}.

\begin{figure}[ht!]
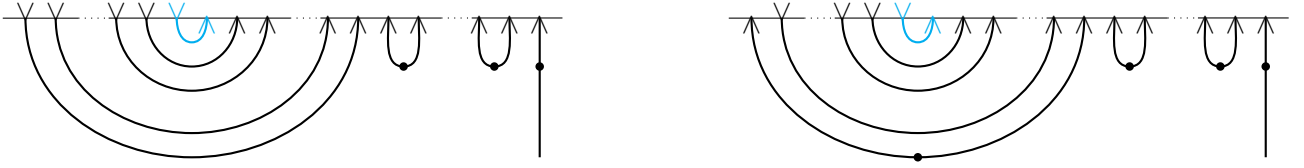

 
\caption{The cup diagrams $\underline{\mu}\mu$ for $\mu=(1,2,\dots,c^k)$ for $c\geq k$ and $c$ even and odd respectively. Note that a decorated strand on the right-hand appears if and only if $c-k$ is odd.} 
\label{Inconcup1}
\end{figure}

\item If $c<k$, then $\cp$ is covered by $c-1$ undecorated cups and there are $k-c-1$ (respectively $k-c$) undecorated strands and $1$ (respectively $0$) decorated strands to the left of $\cp$ when $c$ is even (resp. odd) as pictured in \cref{Inconcup2}
\begin{figure}[ht!]
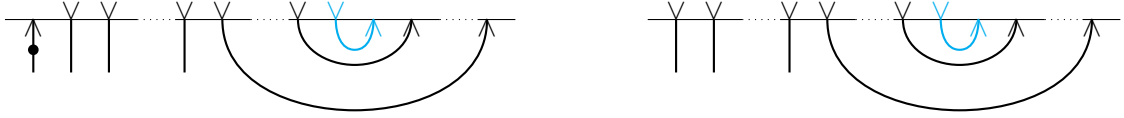

%
 
\caption{The cup diagrams $\underline{\mu}\mu$ for $\mu=(1,2,\dots,c^k)$ for $c< k$ and $c$ even and odd respectively.}
\label{Inconcup2}
\end{figure}
 \end{itemize}

\item If $\cp$ is doubly covered, then $\mu=(1,2,\dots,c)$ and $\cp$. If $c$ is even, then $\cp$ is undecorated, else $\cp$ is decorated as shown in \cref{Inconcup3}. In both case $\cp$  is doubly covered by $\lfloor \frac{c-1}{2}\rfloor$ cups.
\begin{figure}[ht!]
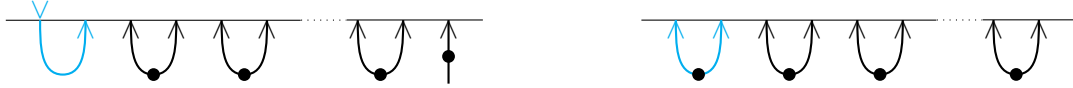

%
 
\caption{The cup diagrams $\underline{\mu}\mu$ for $\mu=(1,2,\dots,c)$ for $c$ even and odd respectively. These are the two incontractible cases when $\cp$ is doubly covered.}
\label{Inconcup3}
\end{figure}

\item Else, if $\cp$ is not covered or doubly covered by any cups, then $\mu=(1),(1,2),(1^2)$ or $(1^c)$ and hence looks like one of three cup diagrams pictured in \cref{Inconcup4}.

\begin{figure}[ht!]
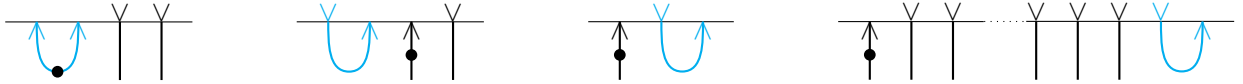

%
 
\caption{The cup diagrams $\underline{\mu}\mu$ for $\mu=(1), \mu=(1,2), \mu=(1^2)$ and $\mu=(1^c)$ respectively.}
\label{Inconcup4}
\end{figure}
\end{itemize}
\end{rmk}

\begin{rmk}
\label{exercise}
\noindent Note that the strands to the left of $\cp$ in \cref{Inconcup2}, when $\cp$ is covered, are unaffected by removing $\cp$ or any cup adjacent to $\cp$. They therefore do not play any significant part in the multiplication. Calculations involving the image of these diagrams in $\Psi$ will essentially be identical to those in the case when $c$ is even and $c\geq k$. So, we often leave those calculations corresponding to the cases pictured in \cref{Inconcup2} as an exercise in what follows.
\end{rmk}

We now proceed to check that the relations \eqref{rel1}--\eqref{adjacentcup} from \cref{presentation} hold in the image of $\Psi$. We begin with the simplest relation.

\begin{prop}[The idempotent relations]
\label{tin pot}
The idempotent relations are preserved by $\Psi$.
Equivalently, $$ 
(\underline{\la}\la\overline{\la})(\underline{\mu}\mu\overline{\mu})=\delta_{\la,\mu} \underline{\la}\la\overline{\la} \qquad 
\qquad  i^{\ka(\la,\mu)}(\underline{\la}\la\overline{\la})(\underline{\la}\la\overline{\mu}) (\underline{\mu}\mu\overline{\mu}) =  i^{\ka(\la,\mu)}(\underline{\la}\la\overline{\mu}).
$$
\end{prop}

\begin{proof}
$\underline{\la}\la\overline{\la}$ only contains small anti-clockwise oriented circles. Thus, the relation follows by the definition of merge surgeries in $\mathbb{D}_n$ and repeatedly applying \cref{idempot lem}.
\end{proof}

\begin{prop}[The commuting relations] 
\label{KcomR}
The commuting relations are preserved by $\Psi$. Equivalently,
suppose that $\cp,\mq\in \underline{\mu}$ commute and  $\la =\mu-\cp, \quad \nu=\mu-\mq, \quad \sigma=\mu-\mq-\cp$.
Then, \begin{align}
\label{Kcom1}
i^{\ka(\sigma,\la)+\ka(\la,\mu)}(\underline{\sigma}\sigma\overline{\la})(\underline{\la}\la\overline{\mu})&=i^{\ka(\sigma,\nu)+\ka(\nu,\mu)}(\underline{\sigma}\sigma\overline{\nu})(\underline{\nu}\nu\overline{\mu})\\
\label{Kcom2}
i^{\ka(\la,\mu)+\ka(\nu,\mu)}(\underline{\la}\la\overline{\mu}) (\underline{\mu}\nu\overline{\nu}) &=i^{\ka(\sigma,\la)+\ka(\sigma,\nu)}(\underline{\la}\sigma\overline{\sigma}) (\underline{\sigma}\sigma\overline{\nu}).
\end{align}
\end{prop}

\begin{proof}
First,  we check the signs in \eqref{Kcom1} and \eqref{Kcom2}. Note that $\sigma= \la-\mq$ and so it is easy to check that $\ka(\sigma,\la)=\ka(\mu,\nu)$. Similarly $\sigma= \nu-\cp$, so $\ka(\sigma,\nu)=\ka(\la,\mu)$ and hence the powers of $i$ on both sides of \cref{Kcom1,Kcom2} are the same.

If either $r_\mq<l_\cp$ or  $r_\cp<l_\mq$, then the relations follow immediately from \cref{idempot lem} since the only two components that aren't small anti-clockwise oriented circles are far away from each other and hence do not interact.

Else, $\cp \prec \mq$ and $\cp$ and $\mq$ commute. In what follows, we only explicitly check \eqref{Kcom2}; \eqref{Kcom1} follows in an almost identical manner and is left as an exercise. Applying \cref{KContraction homomorphism} inductively, we only need to consider the incontractible case where $b(\cp)=1$. Using \cref{Incontractible cup diagrams} we have that case $\mu=(1,2,\dots,c^k)$ and $\la=(1,2,\dots,c^{k-1},c-1)$.
We will also assume that $c\geq k$; the case $c<k$ is similar, as noted above (\cref{exercise}), and is left as an exercise.  There are three cases to consider:

Case $(i)$. If  $c$ and $k$ are both odd, then $\mq$ is decorated and
then $(\underline{\la}\la\overline{\mu}) (\underline{\mu}\nu\overline{\nu})$ is equal to
\begin{equation}
\label{com pic 1}
\begin{minipage}{7cm}
 
\end{minipage}\end{equation}
and $(\underline{\la}\sigma\overline{\sigma}) (\underline{\sigma}\sigma\overline{\nu})$ is equal to:
\begin{equation}
\label{com pic 2}
\begin{minipage}{7cm}%
 
\end{minipage}\end{equation}
where the dashed lines denote $k-2$ small anti-clockwise circles and the equalities follow by applying \cref{idempot lem} $k+\tfrac{k-c}{2}$ times gives the equalities above.

Case $(ii)$.
Now suppose that  $c$ and $k$ are both even, in which case
 $\mq$ becomes undecorated.
The diagrams  
  $ (\underline{\la}\la\overline{\mu})( \underline{\mu}\nu\overline{\nu})$ and 
  $ (\underline{\la}\sigma\overline{\sigma})(\underline{\sigma}\sigma\overline{\nu}) $
   are obtained from those in \eqref{com pic 1} and \eqref{com pic 2} 
 by simply  removing the decoration  from each arc 
   $\mq$ and re-orienting at $l_\mq$ (from $\down$ to $\up$);
   neither of these differences affects the merge surgery defined in \cref{mergeeq}.    
    
Finally, it remains to consider the case that one of $c$ and $k$ is odd and the other even.  If $c$ is odd, then the diagrams are  identical to those of 
Case $(i)$, pictured in   \eqref{com pic 1} and \eqref{com pic 2} with the exception that we add a  decorated strand to the far right-hand side of each; 
similarly if $k$  is odd and $c$ is even,
 then the diagrams are identical to those of   Case $(ii)$ with a decorated strand added to the far right.
In either case, 
these strands do not affect the multiplication, and so we are done.  
\end{proof}

\begin{prop}[The non-commuting relations] 
\label{KCR}
The non-commuting relations are preserved by $\Psi$.
Equivalently, suppose that $\cp,\mq\in \underline{\mu}$ are such that $\mq \prec \cp$ are non-commuting. Then $\mq$ is adjacent to a pair of non-concentric commuting cups, which we label by ${\color{darkgreen}q}^1$ and ${\color{orange}q}^2$ from left to right, in $\underline{\mu-\mq}$ (with ${\color{darkgreen}q}^1$ possibly decorated).
Let $\la =\mu-\cp, \quad \nu=\mu-\mq, \quad \alpha=\mu-\mq-{\color{darkgreen}q}^1, \quad \beta=\mu-\mq-{\color{orange}q}^2$. Then:
\begin{equation}
\label{K-non-com}
i^{\ka(\la,\mu)+\ka(\nu,\mu)}(\underline{\la}\la\overline{\mu})(\underline{\mu}\nu\overline{\nu})=
i^{\ka(\la,\alpha)+\ka(\alpha,\nu)}(\underline{\la}\la\overline{\alpha})(\underline{\alpha}\alpha\overline{\nu})= 
i^{\ka(\la,\beta)+\ka(\beta,\nu)} (\underline{\la}\la\overline{\beta})(\underline{\beta}\beta\overline{\nu})
\end{equation}
\end{prop}

\begin{proof}
We first check that the signs in \eqref{K-non-com} are equal throughout. The cups ${\color{darkgreen}q}^1$ and ${\color{orange}q}^2$ are both adjacent to $\cp$ and $\mq$. Hence we have that $l_\cp=l_{{\color{darkgreen}q}^1}, r_\cp=r_{{\color{orange}q}^2}, l_\mq=r_{{\color{darkgreen}q}^1}$ and $r_\mq=l_{{\color{orange}q}^2}$ (see cases in a) and b) in \cref{adj lem2} or \cite[Remark 5.10]{My1}). Therefore, we obtain that
 \begin{align}
  \label{sign non-com}
 \begin{split}
\ka(\la,\alpha)+\ka(\alpha,\nu) &= \ka(\la,\beta)+\ka(\beta,\nu)\\
&=\tfrac{1}{2}(l_{{\color{darkgreen}q}^1}+r_{{\color{darkgreen}q}^1})+\tfrac{1}{2}(l_{{\color{orange}q}^2}+r_{{\color{orange}q}^2})\\
&=\tfrac{1}{2}(l_\cp+r_\cp)+\tfrac{1}{2}(l_\mq+r_\mq)\\
&=\ka(\la,\mu)+\ka(\nu,\mu)
 \end{split}
 \end{align}
Thus, by applying \cref{KContraction homomorphism} inductively, it suffices to restrict our focus to the incontractible case;
by considering the diagrams in \cref{Incontractible cup diagrams} we realise that we need only consider the case that 
$b(\mq)=1$ and $b(\cp)=2$ with $\mq\prec \cp$. Explicitly, this when $\mu=(1,2,\dots,c^k)$, $\nu=(1,2,\dots,c^{k-1},c-1)$ and $\la=(1,2,\dots,(c-1)^{k-1},c-2)$.

We now demonstrate the product when $k=2$ and $c$ is odd; the other cases will be similar. If $k=2$ and $c$ is odd, then $\cp$ is decorated and 
 $(\underline{\la}\la\overline{\mu})(\underline{\mu}\nu\overline{\nu})$ is equal to
\begin{equation}
\label{non-com pic 1}
\begin{minipage}{5.4cm}
 
\end{minipage},\end{equation}
where the equality follows by $\frac{c+1}{2}$  
applications of \cref{idempot lem}. 
 Similarly,  we have that   $(\underline{\la}\la\overline{\alpha})(\underline{\alpha}\alpha\overline{\nu})$ is equal to
\begin{equation}
\label{noncom pic 2}
\begin{minipage}{5.4cm}%
 
\end{minipage},\end{equation}
simply by $\frac{c+1}{2}$ applications of \cref{idempot lem}; 
finally, we have that  $(\underline{\la}\la\overline{\beta})(\underline{\beta}\beta\overline{\nu})$ is equal to

\begin{equation}
\label{non-com pic 3}
\begin{minipage}{5.4cm}%
 
\end{minipage},\end{equation}
 again by   $\frac{c+1}{2}$ applications of \cref{idempot lem}, which shows that \eqref{K-non-com} holds. 
If $k>2$, $\cp$ is not decorated; it will be covered by  $k-3$ small anti-clockwise oriented undecorated circles and one decorated small anti-clockwise oriented circle, but using \cref{idempot lem}, these don't affect the surgery procedure, and therefore the calculations are essentially identical. 
If $k>c$, there will be strands on the left-hand side of the diagram that similarly have minimal effect on the calculations.
Finally, if $c$ is even, there will be no decorations present, and the calculation is essentially equivalent to the non-commuting case for type $A$ as in \cite[Section 9]{ChrisDyckPaper}.

\end{proof}

\begin{prop}[The doubly non-commuting relations] 
\label{KDC}
The doubly non-commuting relations are preserved by $\Psi$.
Equivalently, suppose that $\cp,\mq\in \underline{\mu}$ are such that $\mq \prec \prec \cp$ and they do not commute. Then $\mq$ is adjacent to a pair of concentric non-commuting cups, ${\color{darkgreen}q}^1$ and ${\color{orange}q}^2$ with $ {\color{darkgreen}q}^1\prec{\color{orange}q}^2 $ in $\underline{\mu - \mq}$ (with ${\color{orange}q}^2$ possibly decorated). 
Let $\la =\mu-\cp, \quad \nu=\mu-\mq, \quad \alpha=\mu-\mq-{\color{orange}q}^2$. Then:
\begin{equation}
\label{K-double-non-com} 
i^{\ka(\la,\mu)+\ka(\nu,\mu)}(\underline{\la}\la\overline{\mu})(\underline{\mu}\nu\overline{\nu})=
i^{\ka(\la,\alpha)+\ka(\alpha,\nu)}(\underline{\la}\la\overline{\alpha})(\underline{\alpha}\nu\overline{\nu})
\end{equation}
\end{prop}

\begin{proof}
We first check that the signs are equal on either side of \eqref{K-double-non-com}.
Just as in non-commuting case, the cups ${\color{darkgreen}q}^1$ and ${\color{orange}q}^2$ are both adjacent to $\cp$ and $\mq$ and so $l_\cp=l_{{\color{orange}q}^2}, r_\cp=r_{{\color{darkgreen}q}^1}, l_\mq=r_{{\color{orange}q}^2}$ and $r_\mq=l_{{\color{darkgreen}q}^1}$ (see cases in c) and d) in \cref{adj lem2} or \cite[Remark 5.10]{My1}). So performing the same calculation as \cref{sign non-com} suffices to show that the signs are equal on both sides of \eqref{K-double-non-com}.

Therefore, as in the non-commuting case, applying inductively \cref{KContraction homomorphism}, we need only consider the incontractible case with $b(\mq)=1$ and $b(\cp)=3$. 
Explicitly, $\mu=(1,2,\dots,c)$, $ \nu=(1,2,\dots,(c-1)^2)$ and $\la=(1,2,\dots,c-2,(c-3)^2)$. If $c$ is odd, then $\mq$ is decorated and 
by applying \cref{idempot lem} $\frac{c+1}{2}$ times, we have that $(\underline{\la}\la\overline{\mu})(\underline{\mu}\nu\overline{\nu})$  is equal to
\begin{equation}
\label{doubly non-com pic 1}
\begin{minipage}{5.3cm}
 
\end{minipage}\end{equation} and 
$(\underline{\la}\la\overline{\alpha})(\underline{\alpha}\alpha\overline{\nu})$ is equal to
\begin{equation}
\label{doubly non-com pic 2}
\begin{minipage}{5.3cm}%
 
\end{minipage}\end{equation} which proves that \cref{K-double-non-com} holds.
If $c$ is even, $\mq$ is then undecorated, but the only difference in the calculations is this change in decoration of certain arcs. These decorations don't affect the merging calculations at all, and the equalities above follow almost identically.
\end{proof}

We now turn to proving that the most computationally difficult relation, the self-dual relation, holds in the image of $\Psi$. Because this requires significantly more effort, we divide the proof into two parts: first showing that \eqref{selfdualrelD} holds in the image of $\Psi$, and then proceeding to do the same for \eqref{selfdualrelA}.
To simplify the signs of \eqref{selfdualrelD} and \eqref{selfdualrelA} in the image of $\Psi$, we will first prove a preliminary lemma.

\begin{lem}
\label{signs}
Suppose that $\nu=\mu-\mq$ is contractible at $k\in\mathbb{Z}_{\geq0}$ and that $\Phi_k(\nu)=\Phi_k(\mu)-\mq'$ Then 
$$ (-1) ^{b(\mq)\pm\ka(\mu,\nu)}=(-1) ^{b(\mq')\pm\ka(\Phi_k(\mu),\Phi_k(\nu))}$$
\end{lem}

\begin{proof}
Suppose $\diamond$ is a node in $\mu$ at position $i-\frac{1}{2}$. If $i<k$ then $\Phi_k(\diamond)=i-\frac{1}{2}$ and if $i>k$ then $\Phi_k(\diamond)=i-2-\frac{1}{2}$. The claim then follows from the fact that $b(\mq)+\ka(\nu,\la)$ is equal to $r_\mq+\frac{1}{2}$ or $r_\mq+l_\mq$, $\mq$ undecorated or decorated respectively (and the obvious equivalent for $b(\mq)-\ka(\nu,\la)$ is similar). We also note that if $k\neq0$ and $l_{\mq'}=\frac{1}{2}$ then the decoration of $\mq$ and $\mq'$ will differ, however in this case the expression for $b(\mq+\ka(\nu,\la)$ is clearly unchanged (and likewise similarly for $b(\mq)-\ka(\nu,\la)$). 
\end{proof}

\begin{prop}[The self-dual relations: part 1] 
\label{selfdual1}
The equation \eqref{selfdualrelD} is preserved by $\Psi$. Equivalently,
suppose that ${\color{cyan}p}\in \underline{\mu}$ and $\la = \mu - {\color{cyan}p}$. If $\cp$ is doubly covered and hence the two cups adjacent to $\cp$, $\ot, \gr\in \underline{\la}$ are non-commuting with $\gr \prec \ot$. Setting $\alpha=\la-\gr,\beta=\la-\ot$ then we have
	\begin{multline}
	(-1)^{\ka(\la,\mu)}(\underline{\la}\la\overline{\mu})(\underline{\mu}\la\overline{\la})
= (-1)^{b({\color{cyan}p})-1}\Bigg(
2
\!\! \sum_{   \begin{subarray}{c} \nu=\la-\mq \\ \mq \in (\underline{\mu} \cap \underline{\la}) \\ \cp \prec\mq,\cp \prec\prec\mq \end{subarray}}
\!\!
(-1) ^{b({\color{magenta}q})+\ka(\nu,\la) } (\underline{\la}\nu\overline{\nu})(\underline{\nu}\nu\overline{\la})+ \\
2  (-1)^{b(\ot)+\ka(\beta,\la) }  (\underline{\la}\beta\overline{\beta})(\underline{\beta}\beta\overline{\la}) + 
 (-1)^{b(\gr)+\ka(\alpha,\la) }  (\underline{\la}\alpha\overline{\alpha})(\underline{\alpha}\alpha\overline{\la})\Bigg),
 \label{SeldualK1}
\end{multline}
where throughout we refer to the set $\underline{\mu} \cap \underline{\la}$ to be the cups in $\underline{\mu}$ that commute with $\cp$ (and hence by \cref{commie2} are in $\underline{\la}$ also) and abbreviate ``adjacent to" simply as ``adj."
\end{prop}

\begin{proof}
Using \cref{signs}  and \cref{KContraction homomorphism}, we can assume that $\underline{\mu} {\la}$ is incontractible in what follows. We begin by calculating the signs in \eqref{SeldualK1}.
By our assumption that $\underline{\mu} {\la}$ is incontractible and $\mu=\la+\cp$,
 we have that   $\mu=(1,2,3,\dots,c)$ and $\la=(1,2,3,\dots,(c-1)^2)$.
 If $c$ is even then $\underline{\mu}\mu$ and $\underline{\la}\la$ are as pictured for reference in \cref{self dual cups2}.
\begin{figure}[ht!]
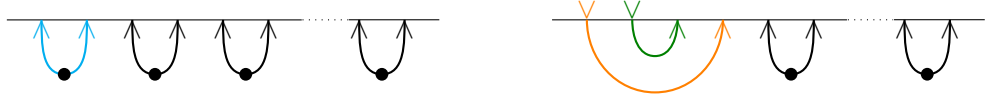

 
\caption{The cups diagrams $\underline{\mu}\mu$ and $\underline{\la}\la$, when $c$ is odd.  (Each diagram has $x=\frac{c+1}{2}$ cups in total.)
}
\label{self dual cups2}
\end{figure}
 If $c$ is odd then $\underline{\mu}\mu$ and $\underline{\la}\la$ are pictured in \cref{self dual cups3}.
\begin{figure}[ht!]
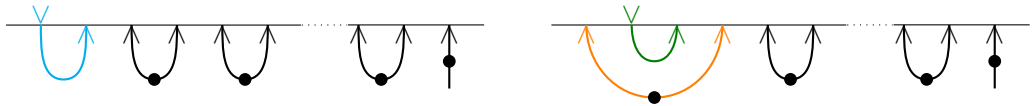

%
 
\caption{The cups diagrams $\underline{\mu}\mu$ and $\underline{\la}\la$, when $c$ is even.  (Both diagrams have $\frac{c}{2}$ cups in total.)
}
\label{self dual cups3}
\end{figure}
 
 We first perform explicit calculations in the case that
  $c=2x-1$ is odd. We will then remark on the small changes needed if $c$ is even at the end of this calculation. To do so, we start by calculating the signs in \eqref{SeldualK1}.
  In $\underline{\mu} $, the cup $\cp$ is doubly covered by a total of 
  $x-1$  decorated cups, $x-2$ of which also appear unchanged in $\underline{\la}$.  
    Hence, the set 
$$\{\mq_j \in (\underline{\mu} \cap \underline{\la}) \mid \cp \prec \prec\mq\}=\{\mq_{3},\dots,\mq_x\}$$ 
is such that $b(\mq_j)=\frac{1}{2}(\frac{4j-3}{2}+\frac{4j-1}{2})=4j-2$ for $3\leq j\leq x$. 
For $3\leq j\leq x$, setting  $\nu_j=\la-\mq_j$, we have that 
 $\ka(\nu_j,\la)=\frac{1}{2}(\frac{4j-3}{2}+\frac{4j-1}{2})=4j-2$ 
%
%
 We also have that $\cp \in \underline{\mu}$
    is adjacent to the two non-commuting cups $\gr,\ot \in\underline{\la}$, where $\gr \prec \ot$;  
we have that $b(\ot)=2, b(\gr)=1$ and that $\ka(\beta,\la)=\ka(\alpha,\la)=2$. 
Finally $\cp$ is the leftmost cup in $\underline{\mu}$  and so $\ka(\la,\mu)=1$. 

  In light of the above, \eqref{SeldualK1} simplifies  to the following 
\begin{equation}
\label{Dsum}
	(-1)(\underline{\la}\la\overline{\mu})(\underline{\mu}\la\overline{\la})
=
2 \sum_{j=3}^{x} (\underline{\la}\nu_j\overline{\nu_j})(\underline{\nu_j}\nu_j\overline{\la})
+2(\underline{\la}\beta\overline{\beta})(\underline{\beta}\beta\overline{\la}) -
 (\underline{\la}\alpha\overline{\alpha})(\underline{\alpha}\alpha\overline{\la}).
\end{equation}
The remainder of the proof is dedicated to checking \cref{Dsum}. 

In what follows we use $D_i$, for $i\in\mathbb{Z}_{\geq0}$, to denote the diagram $\underline{\la}(\la-{\color{blue}s}_i)\overline{\la}$ where ${\color{blue}s}_i$ is the cup satisfying $l_{{\color{blue}s}_i}=i$. Equivalently, $D_i$ is the diagram $\underline{\la}\la\overline{\la}$ with the small circle passing through node $i$ changed to be oriented clockwise.

\noindent{\bf The left-hand side of \cref{Dsum}. }
Inputting the half diagrams pictured in  \cref{self dual cups2} into the left-hand side of \cref{Dsum}, we obtain 
 \begin{align}  
 \label{Picpic1}
 &(\underline{\la}\la\overline{\mu})(\underline{\mu}\la\overline{\la})=
 \begin{minipage}{4cm}
 
\end{minipage}.
\intertext{Applying 
$(\tfrac{1}{2},\tfrac{3}{2})$-surgery to the diagram on the right of \eqref{Picpic1}, using the split surgery rules as given in \cref{spliteq}, we obtain 
}
&=(-1)^0\Bigg((-1)^1
\begin{minipage}{3.8cm}
%
 
\end{minipage}\Bigg)  =-D_2-D_1,
\label{Dsumboi}
\end{align}
 where in the 
  final equality we have used $D_i$ as defined above.

\smallskip
\noindent{\bf The right-hand side of  \cref{Dsum}. }
We proceed term-by-term on the right-hand side of \eqref{Dsum}, beginning with the terms associated with cups $\ps$ having the smallest $r_\ps$, and proceeding in increasing order of $r_\ps$.
We first calculate the terms \eqref{Dsum} labelled by the adjacent cups $\gr$ and $\ot$.
First by applying \cref{idempot lem} $x-1$ times we obtain that
 \begin{align}
 \label{picpicpic}
 & (\underline{\la}\la\overline{\alpha})(\underline{\alpha}\la\overline{\la})=
 \begin{minipage}{4cm}
 
\end{minipage}.
\intertext{Applying 
$(\tfrac{1}{2},\tfrac{3}{2})$-surgery to the right-hand side of \cref{picpicpic}, as in \eqref{spliteq} we obtain
}
&(-1)^0\Bigg((-1)
\begin{minipage}{4.5cm}
%
 
\end{minipage}\Bigg)=-D_{2}+D_{1}.
\label{Dsumboiadj}
\end{align}

Next, applying \cref{idempot lem} $x-1$ times, we obtain that
 \begin{align} 
 \label{picsallday}
 & (\underline{\la}\beta\overline{\beta})(\underline{\beta}\beta\overline{\la})=
 \begin{minipage}{5cm}
%
 
\end{minipage}.
\intertext{Performing $(\tfrac{7}{2},\tfrac{9}{2})$-surgery to the right-hand side of \eqref{picsallday} we obtain}
&(-1)^3\Bigg(
\begin{minipage}{4.5cm}
%
 
\end{minipage}\Bigg)=-D_{5}+D_{1}.
\label{Dsumboiadj2}
\end{align}

\noindent We now consider the terms from the sum on the right-hand side of \eqref{Dsum}.  
We start with the terms corresponding to all but the rightmost double covering cups, $\nu_j$ for $3\leq j \leq  x-1 $.
So for  $3\leq j \leq  x-1 $, by applying \cref{idempot lem} $x-1$ times we have that
 \begin{align}\label{pick}
 &(\underline{\la}\nu_j\overline{\nu_j})(\underline{\nu_j}\nu_j\overline{\la})=
 \begin{minipage}{5cm}
%
 
\end{minipage}.
\intertext{ Applying $(\tfrac{4j-1}{2},\tfrac{4j+1}{2})$-surgery to the right-hand side of \eqref{pick} we obtain:}
&(-1)^{2j-1}\Bigg(
\begin{minipage}{5.2cm}
%
 
\end{minipage}\Bigg)\nonumber\\
&=(-1)^{2j-1}(D_{2j+1}-D_{2j-1})=-D_{2j+1}+D_{2j-1}.
\label{DsumboiDC}
\end{align}

\noindent Finally, we consider the term corresponding to the final double covering cup $\nu_x$. Simply by applying \cref{idempot lem} $x-1$ times we have that 
\begin{equation}
(\underline{\la}\nu_x\overline{\nu_x})(\underline{\nu_x}\nu_x\overline{\la})=
 \label{DsumboiDCend}
 \begin{aligned}
 & 
 \begin{minipage}{4cm}
%
 
\end{minipage}=D_{2x-1}.
\end{aligned}
\end{equation}

\smallskip
\noindent{\bf Reconciling the left and right-hand sides of \cref{Dsum}. }
We plug \eqref{Dsumboi}, \eqref{Dsumboiadj}, \eqref{Dsumboiadj2},  \eqref{DsumboiDC}, \eqref{DsumboiDCend} into \eqref{Dsum} to obtain 
\begin{equation*}\begin{aligned}
\label{Dsumfinal}
&2 \sum_{j=3}^{x-1} (\underline{\la}\nu_j\overline{\nu_j})(\underline{\nu_j}\nu_j\overline{\la})+2(\underline{\la}\nu_{x}\overline{\nu_{x}})(\underline{\nu_{x}}\nu_{x}\overline{\la}))
+2(\underline{\la}\beta\overline{\beta})(\underline{\beta}\beta\overline{\la}) 
- (\underline{\la}\alpha\overline{\alpha})(\underline{\alpha}\alpha\overline{\la}).\\
 &=2 \sum_{j=3}^{x-1} (D_{2j-1}-D_{2j+1})+2D_{2x-1}
 +2(D_{1}-D_{5}) -(D_{1}-D_{2})\\
 &=2(D_{5}-D_{2x-1})+2D_{2x-1}+D_1-2D_5+D_2\\
 &=D_2+D_1=-(\underline{\la}\la\overline{\mu})(\underline{\mu}\la\overline{\la}),
 \end{aligned}
\end{equation*}
as required. Hence \eqref{SeldualK1} holds for $c$ odd.

Finally, we discuss the minor changes required in the case that $c$ is even (the corresponding cup diagrams for $\underline{\mu}\mu$ and $\underline{\la}\la$ are pictured in \cref{self dual cups3}). 
If $c$ is even, 
 the cup $\cp$ is undecorated, and there is a decorated strand to the right of the rightmost double covering cup $\mq_x$. 
This change affects the surgery, but {\em only by changing the signs} involved.
Repeating the above calculation with these sign changes, the result follows.
Specifically, the final term associated with the rightmost double covering cup $\mq_x$ is resolved via a split surgery. An equivalent example of this can be found in the proof \cref{selfdual2}, where these terms of the alternating sum cancel in an analogous manner.
While the case calculated in that proof assumes $c$ is odd and $k$ is even, the decorated strand to the right $\mq_x$, is the reason the cancellation is equivalent.
 \end{proof}

\begin{prop}[The self-dual relations: part 2] 
\label{selfdual2}
The equation \eqref{selfdualrelA} is preserved by $\Psi$. Equivalently,
let  ${\color{cyan}p}\in \underline{\mu}$ and $\la = \mu - {\color{cyan}p}$.  
If $\cp$ doesn't satisfy the conditions in \cref{selfdual1} then we have 
\begin{multline}
	(-1)^{\ka(\la,\mu)}(\underline{\la}\la\overline{\mu})(\underline{\mu}\la\overline{\la})
= (-1)^{b({\color{cyan}p})-1}\Bigg(
2
\!\! \sum_{   \begin{subarray}{c} \nu=\la-\mq \\ \mq \in (\underline{\mu} \cap \underline{\la}) \\ \cp \prec\mq,\cp \prec\prec\mq \end{subarray}}
\!\!
(-1) ^{b({\color{magenta}q})+\ka(\nu,\la) } (\underline{\la}\nu\overline{\nu})(\underline{\nu}\nu\overline{\la})\\
+\!\!\sum_{  \begin{subarray}{c} \sigma=\la-\gr  \\ \gr \, \text{adj.}\, {\color{cyan}p} \end{subarray}}
\!\!
 (-1)^{b(\gr)+\ka(\sigma,\la) }  (\underline{\la}\sigma\overline{\sigma})(\underline{\sigma}\sigma\overline{\la})\Bigg).
 \label{selfdualrelK2}
\end{multline}  
where throughout we refer to the set $\underline{\mu} \cap \underline{\la}$ to be the cups in $\underline{\mu}$ that commute with $\cp$ (and by \cref{commie2} are in $\underline{\la}$ also) and abbreviate ``adjacent to" simply as ``adj."
\end{prop}

\begin{proof}
By \cref{signs}  and \cref{KContraction homomorphism}, we can assume that $\underline{\mu} {\la} $ is incontractible in what follows.  We begin by calculating the signs in \eqref{selfdual2}.
By our assumption $\underline{\mu} {\la}$ is incontractible and $\mu=\la+\cp$, we have that $\mu=(1,2,3,\dots,c^k)$ for some $k$ and $b(\cp)=1$
and $\la=(1,2,3,\dots,c^{k-1},c-1)$.
 If $c$ is odd, $k$ is even and $c\geq k$ then $\underline{\mu}\mu$ and $\underline{\la}\la$ are as pictured in \cref{self dual cups}.
 If $c$ is even and $k$ is odd or $c< k$ then $\underline{\mu}\mu$ is pictured in \cref{Inconcup2,Inconcup1}.

\begin{figure}[ht!]
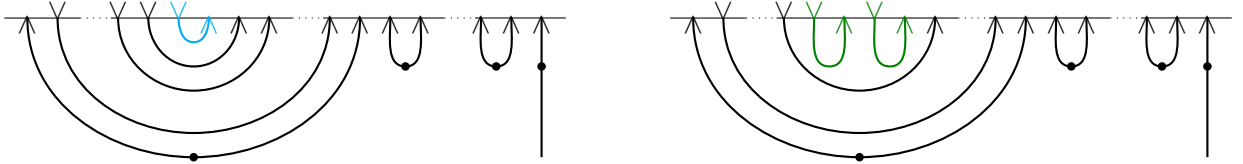

 
\caption{The cups diagrams $\underline{\mu}\mu$ and $\underline{\la}\la$, when $c$ is odd and $c>k$. (Each diagram has $x'=\frac{c+1}{2}$ cups.)}
\label{self dual cups}
\end{figure}

 We will first perform explicit calculations in the case that $c> k$, $k$ is even and that
  $c=2x'-1$ is odd, and we will remark on the small changes needed in the other cases at the end of the calculation. 
 In $\underline{\mu}$ the cup 
$\cp$ is covered by $k-1$ cups, $k-2$ of which appear unchanged in $\underline{\la}$. All of these $k-2$ cups are undecorated except for the largest, $\mq_k$, with $l_{\mq_k}=\frac{1}{2}, r_{\mq_k}=\frac{4k-1}{2}$.
Hence, the set 
$$\{\mq_j \in (\underline{\mu} \cap \underline{\la}) \mid \cp \prec\mq\}=\{\mq_3,\dots,\mq_k\}$$ 
satisfies $b(\mq_j)=\frac{1}{2}(\frac{2k+2j-1}{2}-\frac{2k-2j+1}{2}+1)=j$ for $3\leq j\leq k$.
Setting $\nu_j=\la-\mq_j$, we also have that, for $3\leq j\leq k$, $\ka(\nu_j,\la)=k=\ka(\la,\mu)$.
Additionally in $\underline{\mu}$ the cup 
$\cp$ is  doubly covered by $x'-k$ decorated cups all of which appear unchanged in $\underline{\la}$. 
Therefore, the set
$$\{\mq_j \in (\underline{\mu} \cap \underline{\la}) \mid \cp \prec \prec\mq\}=\{\mq_{k+1},\dots,\mq_{x'}\}$$ 
is such that $b(\mq_j)=\frac{1}{2}(\frac{4j-3}{2}+\frac{4j-1}{2})=2j-1$ for $k+1\leq j\leq x'$.
Using the same notation as above, we set  $\nu_j=\la-\mq_j$ and we have that, for $k+1\leq j\leq x'$, $\ka(\nu_j,\la)=\frac{1}{2}(\frac{4j-3}{2}+\frac{4j-1}{2})=2j-1$.
The cup 
$\cp$ is adjacent, from left to right, to two commuting cups $\gr_1,\gr_2\in \underline{\la}$. 
We have that $b(\gr_1)=b(\gr_2)=1$ and
setting 
$\sigma_j=\la-\gr_j$, we have that  $\ka(\sigma_1,\la)=\frac{1}{2}(l_\gr+r_\gr)=\frac{1}{2}((l_\cp-1)+l_\cp)=\frac{1}{2}((r_\cp-1)-1+l_\cp)=\ka(\la,\mu)-1$ and (similarly) $\ka(\sigma_2,\la)=\ka(\la,\mu)+1$.
Finally,  $\ka(\la,\mu)=k$.
     In light of all the above, \eqref{selfdualrelK2} simplifies to the following 
 \begin{equation}
	(\underline{\la}\la\overline{\mu})(\underline{\mu}\la\overline{\la})
=
2 \sum_{j=3}^k (-1)^{j} (\underline{\la}\nu_j\overline{\nu_j})(\underline{\nu_j}\nu_j\overline{\la})+2(-1)^{k} \sum_{j=k+1}^{x'}(\underline{\la}\nu_j\overline{\nu_j})(\underline{\nu_j}\nu_j\overline{\la})+
\sum_{j=1}^2 
(\underline{\la}\sigma_j\overline{\sigma_j})(\underline{\sigma_j}\sigma_j\overline{\la}) 
 \label{SeldualK3}
\end{equation}
The remainder of the proof is dedicated to checking \eqref{SeldualK3}.

\noindent{\bf The left-hand side of \cref{SeldualK3}. }
Inputting the half diagrams pictured in  \cref{self dual cups} into the left-hand side of \eqref{SeldualK3} we obtain that $(\underline{\la}\la\overline{\mu})(\underline{\mu}\la\overline{\la})$ is equal to
\begin{align} 
\label{pictureofC}
&\begin{minipage}{7.4cm}
 
\end{minipage},
%
%
%
%
%
%
%
%
%
%
%
\intertext{where the equality follows from 
applying  \cref{idempot lem} $x'-1$ times. Next, applying 
$(\frac{2k-1}{2},\frac{2k+1}{2})$-surgery to the diagram on the right of \eqref{pictureofC}, as in \cref{spliteq}, we obtain 
}
 &(-1)^{k-1}\Bigg( (-1)^1
\begin{minipage}{5.8cm}
%
 
\end{minipage} 
\Bigg)\nonumber\\
&
= (-1)^{k}(D_{k+1}+D_{k-1}),
\label{sumboiL}
\end{align}
where we have used the same $D_i$ notation as defined during the proof of \cref{selfdual1}.

\smallskip
\noindent{\bf The right-hand side of  \cref{SeldualK3}. }
We proceed term-by-term on the right-hand side of \cref{SeldualK3}, beginning with the terms associated to cups $\ps$ with minimal $r_\ps$, and proceeding in increasing order of $r_\ps$.
We first start with the terms labelled by the adjacent cups $\gr_1, \gr_2$, where $\sigma_i=\la-\gr_i$. First, we obtain that 
 \begin{align} 
 \label{adj1SDpic}
 &(\underline{\la}\sigma_1\overline{\sigma_1})(\underline{\sigma_1}\sigma_1\overline{\la})=\begin{minipage}{6cm}
 
\end{minipage}
\intertext{where the equality comes from applying \cref{idempot lem} $x'-1$ times. Performing $(\frac{2k-5}{2},\frac{2k-3}{2})$-surgery to the diagram on the right-hand side of \eqref{adj1SDpic}, we obtain}
&(-1)^{k-3}\Bigg((-1)
\begin{minipage}{5.7cm}
%
 
\end{minipage}\Bigg)\nonumber\\
&=(-1)^{k-2}(D_{k-1}-D_{k-2}).
\label{sumboi1}
\end{align}

\noindent and an almost identical calculation shows that:
 \begin{align} 
 &(\underline{\la}\sigma_2\overline{\sigma_2})(\underline{\sigma_2}\sigma_2\overline{\la})=(-1)^{k+1}\Bigg(
\begin{minipage}{5.5cm}
%
 
\end{minipage}\Bigg)\nonumber\\
&=(-1)^{k+1}(D_{k-2}-D_{k+1}).
  \label{sumboi2}
 \end{align}

\noindent Next, we'll calculate each of the terms of the right-hand side of \cref{SeldualK3} labelled by the covering cups, $\mq_j$ for $3\leq j\leq k-2$, where $\nu_j=\la-\mq_j$. By applying \cref{idempot lem} $x'-1$ times, we have that for $3\leq j \leq k-2$, $(\underline{\la}\nu_j\overline{\nu_j})(\underline{\nu_j}\nu_j\overline{\la})$ is equal to:
\begin{align}
\label{coverforme}
&\begin{minipage}{8cm}
%
 
\end{minipage}.
\intertext{Then by performing $(\frac{2k-2j-1}{2},\frac{2k-2j+3}{2})$-surgery to the diagram on right-hand side of \eqref{coverforme} we obtain}
&(-1)^{k-j-1}\Bigg(-
\begin{minipage}{7.4cm}
%
 
\end{minipage}\Bigg)\nonumber\\
&=(-1)^{k-j-1}(-D_{k-j+1}+D_{k-j}).
\label{sumboi4}
\end{align}

\noindent Now, we treat each of the terms of the right-hand side of \cref{SeldualK3} corresponding to the final two covering cups, $\mq_j$ for $j=k-1,k$, where $\nu_j=\la-\mq_j$. These calculations differ slightly from the previous terms. First, the $j=k-1$ term involves a decorated split as the final surgery; by applying \cref{idempot lem} $x'-1$ times $(\underline{\la}\nu_{k-1}\overline{\nu_{k-1}})(\underline{\nu_{k-1}}\nu_{k-1}\overline{\la})$ is equal to:
\begin{align}
\label{bigcoveringpics2}
&\begin{minipage}{7.5cm}
 
\end{minipage}.
\intertext{Then performing $(\frac{1}{2},\frac{3}{2})$-surgery to the diagram on the right-hand side of \eqref{bigcoveringpics2} gives}
&=(-1)^{0}\Bigg((-1)^1
\begin{minipage}{6.5cm}
%
 
\end{minipage}\Bigg)\nonumber\\
&=-D_2+D_{1}.
\label{sumboiedgecase}
\end{align}

\noindent We now calculate the $j=k$ term associated to the largest covering cup, $\nu_k$. Following the same process as above, applying \cref{idempot lem} $x'-1$ times, we get that $(\underline{\la}\nu_k\overline{\nu_k})(\underline{\nu_k}\nu_k\overline{\la})$ is equal to:
\begin{align}
\label{bigcoveringpics}
&\begin{minipage}{7.5cm}
%
 
\end{minipage}.
\intertext{Then performing $(\frac{4k-1}{2},\frac{4k-1}{2})$-surgery to the diagram on the right-hand side of \eqref{bigcoveringpics} gives}
&=(-1)^{2k-1}\Bigg(
\begin{minipage}{6.5cm}
%
 
\end{minipage}\Bigg)\nonumber\\
&=(-1)^{2k-1}(D_{2k+1}-D_{1}).
\label{sumboi}
\end{align}

\noindent Finally, we consider the terms labelled by doubly-covering cups, $\mq_j\in \underline{\mu}$ for $k+1\leq j\leq x'-1$, where $\nu_j=\la-\mq_j$. For $k+1\leq j\leq x'-1$, by applying \cref{idempot lem} $x'-1$ times we obtain that, $(\underline{\la}\nu_j\overline{\nu_j})(\underline{\nu_j}\nu_j\overline{\la})$ is equal to:
\begin{align}
\label{doublecoversmajority}
&\begin{minipage}{8.3cm}
%
 
\end{minipage}
\intertext{Performing $(\frac{4j-1}{2},\frac{4j+1}{2})$-surgery to the diagram on the right-hand side of \eqref{doublecoversmajority} we obtain}
&(-1)^{2j-1}\Bigg(
\begin{minipage}{6.6cm}
%
 
\end{minipage}\Bigg)\nonumber\\
&=(-1)^{2j-1}(D_{2j+1}-D_{2j-1}).
\label{sumboi3}
\end{align}

\noindent We now consider the final term, associated to the cup $\nu_{x'}\in \underline{\mu}$. There are no cups to the right of $\mq_{x'}$ and so we apply \cref{idempot lem} $x'-1$ times, merging small anti-clockwise circles, and then perform $(\frac{4x'-1}{2},\frac{4x'+1}{2})$-surgery to see that $(\underline{\la}\nu_j\overline{\nu_j})(\underline{\nu_j}\nu_j\overline{\la})$ is equal to
\begin{align}
&\begin{minipage}{7.3cm}
%
 
\end{minipage}\Bigg)\nonumber\\
&=D_{2x'-1},
\label{oneterm}
\end{align}
where the $0$ in the first equality comes from the fact that one cannot orient a line clockwise.

\smallskip
\noindent{\bf Reconciling the left and right-hand sides of \cref{SeldualK3}. }
We plug the terms from \eqref{sumboiL}, \eqref{sumboi1}, \eqref{sumboi2}, \eqref{sumboi4}, \eqref{sumboiedgecase}, \eqref{sumboi}, \eqref{sumboi3} and \eqref{oneterm} into \eqref{SeldualK3} to obtain:
\begin{equation*}\begin{aligned}
&2 \sum_{j=3}^k (-1)^{j} (\underline{\la}\nu_j\overline{\nu_j})(\underline{\nu_j}\nu_j\overline{\la})+2(-1)^{k} \sum_{j=k+1}^{x'}(\underline{\la}\nu_j\overline{\nu_j})(\underline{\nu_j}\nu_j\overline{\la})+
\sum_{j=1}^2 
(\underline{\la}\sigma_j\overline{\sigma_j})(\underline{\sigma_j}\sigma_j\overline{\la}) \\
=&2\Big(\sum_{j=3}^{k-2} (-1)^{j} (-1)^{k-j-1}(-D_{k-j+1}+D_{k-j})\Big) +2(-1)^{k-1}(-D_2+D_1)+ 2(-1)^{3k-1}(D_{2k+1}-D_1) \\
&+2\Big((-1)^{k}\sum_{j=k+1}^{x'-1}(-1)^{2j-1}(D_{2j+1}-D_{2j-1})+(-1)^{k}D_{2x'-1}\Big)
+(-1)^{k-2}(D_{k-1}-D_{k-2}) + (-1)^{k+1}(D_{k-2}-D_{k+1})\\
=&(-1)^{k-1}\Bigg(2\Big( \sum_{j=3}^{k-1}(D_{k-j}-D_{k-j+1})\Big) + 2(D_{2k+1}-D_1) \\
&+2\Big(\sum_{j=k+1}^{x'-1}(D_{2j+1}-D_{2j-1})\Big)+2D_{2x'-1}
-(D_{k-1}-D_{k-2}) +(D_{k-2}-D_{k+1})\Bigg)\\
=&(-1)^{k-1}\Big(2D_{1}-2D_{k-2} + 2D_{2k+1}-2D_1
 +2D_{2x'-1}-2D_{2k+1}-2D_{2x'-1}\\
&-D_{k-1}+D_{k-2} +D_{k-2}-D_{k+1}\Big)\\
=&(-1)^{k}(D_{k+1}+D_{k-1})=(\underline{\la}\la\overline{\mu})(\underline{\mu}\la\overline{\la})
\end{aligned}
\end{equation*}
as required and hence \cref{selfdualrelK2} holds for $c> k$, $c$ odd and $k$ even.
  
We now discuss the minor changes required if $c$ is even, if $c$ and $k$ have the same parity, or if $c \leq k$.
\begin{itemize}
\item If $c< k$: The final diagram has $k-c$ rays on the left-hand side, which do not materially affect the multiplication. 
\item If $c=k$, There are no cups that double cover $\cp$. In this case, it is straightforward to check that \eqref{sumboi} is equal to $D_1$ (for the same reason that \eqref{oneterm} is equal to $D_{2x'-1}$). The resulting alternating sum cancels just as above.
\item If $c=2x$ is even (and $k$ is odd): The cup diagrams $\underline{\mu}\mu$ and $\underline{\la}\la$ are as pictured on the left side of \cref{Inconcup1}. 
\item If $c$ and $k$ have the same parity: There is no decorated strand to the right of the rightmost cup $\mq_{x'}\in \mu$ doubly covering $\cp$. 
\end{itemize}
The modifications in the latter two cases affect the surgery calculations, but {\em only by changing the signs}, and the resulting alternating sum cancels similarly. In particular, if $c$ and $k$ have the same parity, because there is no decorated strand to the right, the final terms of the alternating sum cancel in a manner analogous to the odd $c$ case detailed in \cref{selfdual1}.

\smallskip 
\noindent{\bf Degenerate cases. } We now deal with the remaining incontractible cases. These are the cases that are pictured in \cref{Inconcup4}, for which $\cp\in\underline{\mu}$ isn't covered or doubly covered. 
If we are one of the other cases pictured in \cref{Inconcup4}, then $\cp$ is adjacent to one cup. If $\mu=(1,2)\in\mathscr{P}_4$ then the left-hand side of \eqref{selfdualrelK2} becomes
\begin{equation}
 (-1)^0(\underline{\la}\la\overline{\mu})(\underline{\mu}\la\overline{\la})=
 \begin{minipage}{1.8cm}
 
\end{minipage}-0\Bigr)=D_2,
\label{baby1}
\end{equation}
and the right-hand side of \eqref{selfdualrelK2} becomes

\begin{equation} (-1)^2(\underline{\la}\la\overline{\sigma})(\underline{\sigma}\la\overline{\la})
=\begin{minipage}{1.8cm}
				%
 
\end{minipage}-0\Bigr)=D_2.
\label{baby2}
\end{equation}
noting that in both cases, we pick up a term equal to $0$ because the line could not be oriented clockwise, and hence \eqref{baby1} and \eqref{baby2} are equal, as required.

If $\mu=(1^c)$, $c\geq 2$, then the result follows from a similar calculation splitting a propagating strand.
Finally, it is also easy to check that if $\cp$ isn't adjacent to any cups (equivalently $\mu=(1)$), then performing surgery on $(\underline{\la}\la\overline{\mu})(\underline{\mu}\la\overline{\la})$ involves reconnecting two non-propagating strands which is defined to be $0$ using \cref{reconeq}.
This concludes all the possible cases and hence \eqref{selfdualrelK2} holds as claimed.
\end{proof}

Finally, we will check that the last relation, the adjacent relation, is also preserved by $\Psi$. As with the previous relations, it will be significantly easier to perform all calculations on incontractible diagrams of minimal rank. In order to be able to do this, we will need to first prove a few preliminary results relating to the powers of $i$ in \eqref{adjacentcup}. We start by defining a constant for a pair of cups related to these terms.

\begin{lem}
\label{signsanothergo}
Let $\mu,\la,\alpha \in\mptn$ be such that $\la=\mu-\cp$ and $\alpha=\mu-\mq$ for two cups $\cp,\mq\in\underline{\mu}$. 
Define $$C(\cp,\mq)\coloneqq 2\biggl(2b(\mq)+ \ka(\alpha,\mu)-\ka(\la,\mu)\biggr).$$ 
Then we have that
\begin{equation}
\label{sgn21}
 C(\cp,\mq)= 
\left\{ \begin{array}{ll} 
(3r_{\mq}-l_{\mq}-r_\cp-l_\cp)+2&
 \mbox{if $\mq$ is undecorated,
 }
  \\ 
  (3r_{\mq}+3l_{\mq}-r_\cp-l_\cp)
& \mbox{if $\mq$ is decorated.
} 
  \end{array} \right.
\end{equation}
\end{lem}

\begin{proof}
This follows immediately from the defintions of breadth and $\ka(\la,\mu)$ (and $\ka(\alpha,\mu)$).
\end{proof}
In what follows, it will frequently be necessary to track whether a cup is decorated. To this end, we define
$$
\delta_{\text{dec},\cp}=\begin{cases}
1 &\mbox{\text{if $\cp$ is decorated}}\\
0 &\mbox{\text{if $\cp$ is undecorated.}}
\end{cases}
$$
We're interested in this constant $C(\cp,\mq)$ specifically because of what it tells us about the signs in \eqref{K adj}. Consequently, we first consider the difference between values of $C(\cp,\mq)$ and $C(\gr,\ot)$, where $\gr,\ot\in\underline{\mu}$ are adjacent to a cup $\cp\in\underline{\mu}$.
\begin{lem}
\label{signsanothergo222}
Let $\mu, \la \in \mptn$ be such that $\la=\mu-\cp$ for $\cp\in\underline{\mu}$. Let $\gr,\ot \in \underline{\la}$ be adjacent to $\cp$.
Additionally suppose that $\langle \cp\cup \ot\rangle_\mu$ exists (and hence $\underline{\mu}\mu$ and $\underline{\la}\la$ are as pictured in \cref{adj lem2}).
Let $x_1< x_2< x_3< x_4\in\mathbb{Z}_{\geq0}+\frac{1}{2}$ be the x-coordinates of the vertices of $\cp,\langle \cp\cup \ot\rangle_\mu\in\underline{\mu}$ (and hence of $\gr,\ot\in\underline{\la}$).

Then we have that:
\begin{equation}
\label{sgn222}
 C(\cp,\langle \cp\cup \ot\rangle_\mu)-C(\gr,\ot)= 4\biggl(x_1\alpha_{x_1}+x_2\alpha_{x_2}+x_3\alpha_{x_3}+x_4\alpha_{x_4}\biggr)+2(\delta_{\text{dec},\ot}-\delta_{\text{dec},\langle \cp\cup \ot\rangle_\mu}),
\end{equation}
where the $\alpha_i$'s are given as follows:

\begin{minipage}{.55\textwidth}
  \begin{align}
  \label{hello111111}
&\alpha_{x_1}=
\begin{cases}
1 & \mbox{if $l_\ot=r_\cp$ and $\delta_{\text{dec},\ot}\neq\delta_{\text{dec},\langle \cp\cup \ot\rangle_\mu}$},\\
-1 & \mbox{if  $l_\ot=l_\cp$ and $\delta_{\text{dec},\ot}=\delta_{\text{dec},\langle \cp\cup \ot\rangle_\mu}$},\\
0&\mbox{o/w,}
\end{cases} \\
\label{hellox222222}
    &\alpha_{x_2}=
\begin{cases}
-1 & \mbox{if $r_\ot=l_\cp$},\\
0&\mbox{o/w,}
\end{cases} 
  \end{align}
\end{minipage}
\quad\qquad
\begin{minipage}{.4\textwidth}
  \begin{align}
&\alpha_{x_3}=
\begin{cases}
1 & \mbox{if $l_\ot=l_\cp$ },\\
0&\mbox{o/w,}
\end{cases} \\
\label{hello5555}
    &\alpha_{x_4}=
\begin{cases}
1 & \mbox{if $r_\ot=l_\cp$},\\
0&\mbox{o/w.}
\end{cases} 
  \end{align}
\end{minipage}

\end{lem}

\begin{proof}
This follows by checking all the possible adjacent pairs from \cref{adj lem2} for which $\langle \cp\cup \ot\rangle_\mu\in\underline{\mu}$ is well-defined and calculating the respective coefficients. We emphasise that $r_\cp$ and $r_\ot$ are never the shared common vertex between the cups $\cp$ and $\ot$. Furthermore, the coefficient of each $x_i$ term is a multiple of 4 because the values in \eqref{sgn21} are either $3$ or $-1$.
\end{proof}

Next, we consider how \eqref{sgn222} behaves under contraction.

\begin{lem}
\label{signsforthefinaltime}

Let $\mu, \la \in \mptn$ be such that $\la=\mu-\cp$ for $\cp\in\underline{\mu}$. Let $\gr,\ot \in \underline{\la}$ be adjacent to $\cp$.
Additionally suppose that $\langle \cp\cup \ot\rangle_\mu$ exists (and hence $\underline{\mu}\mu$ and $\underline{\la}\la$ are as pictured in \cref{adj lem2}).
Let $x_1< x_2< x_3< x_4\in\mathbb{Z}_{\geq0}+\frac{1}{2}$ be the x-coordinates of the vertices of $\cp,\langle \cp\cup \ot\rangle_\mu\in\underline{\mu}$ (and hence of $\gr,\ot\in\underline{\la}$). 

Additionally, suppose that $\la=\mu-\cp$ is contractible at $k\in\mathbb{Z}_{\geq0}$. We then have that $\Phi_k(\la)=\Phi_k(\mu)-\cp'$, with $\cp' \in\underline{\Phi_k(\mu)}$ adjacent to $\gr',\ot'\in\underline{\Phi_k(\la)}$. We also have that $\langle \cp'\cup \ot'\rangle_{\Phi_k(\mu)}$ exists.
Let $\Phi_k(x_1)<\Phi_k(x_2)<\Phi_k(x_3)<\Phi_k(x_4)\in\frac{1}{2}+\mathbb{Z}_{\geq0}$ be the x-coordinates of the vertices of $\cp',\langle \cp'\cup \ot'\rangle_{\Phi_k(\mu)}\in\underline{\Phi_k(\mu)}$ (and hence of $\gr',\ot'\in\underline{\Phi_k(\la)}$ also). 

$(i)$
  If $\Phi_k(x_1),\Phi_k(x_2),\Phi_k(x_3),\Phi_k(x_4)\neq\frac{1}{2}$ or $k=0$,
then we have that:
\begin{equation}
\label{hello2}
\begin{aligned}
\alpha_{x_1}&=\alpha_{\Phi_k({x_1})}\\
\alpha_{x_2}&=\alpha_{\Phi_k({x_2})}\\
\alpha_{x_3}&=\alpha_{\Phi_k({x_3})}\\
\alpha_{x_4}&=\alpha_{\Phi_k({x_4})}\\
2(\delta_{\text{dec},\ot}- \delta_{\text{dec},\langle \cp\cup \ot\rangle_\mu})&=2(\delta_{\text{dec},\ot'}-\delta_{\text{dec},\langle \cp'\cup \ot'\rangle_{\Phi_k(\mu)}})
\end{aligned}
\end{equation}

$(ii)$ Consequently, we then obtain that $$ C(\cp,\langle \cp\cup \ot\rangle_\mu)-C(\gr,\ot)\equiv C(\cp',\langle \cp'\cup \ot'\rangle_{\Phi_k(\mu)})-C(\gr',\ot')\pmod 8.$$
\end{lem}

\begin{proof}
First, we note that the common vertex between the cups $\cp$ and $\ot$ is unchanged under the contraction mapping (the cups $\cp'$ and $\ot'$ have the same vertex in common).
By \cref{breadth}, we have that unless $l_{\ps'}=\frac{1}{2}$ or $k=0$,  $\delta_{\text{dec},\ps}=\delta_{\text{dec},\ps'}$. If $\delta_{\text{dec},\ps}=\delta_{\text{dec},\ps'}$, the specific case of adjacency listed in \cref{adj lem2} is preserved, meaning that the terms in \eqref{hello111111}-\eqref{hello5555} are unchanged under the contraction map; claim $(i)$ then follows.

Next, suppose $x_i=i-\frac{1}{2}$ in $\mu$. If $i<k$ then $\Phi_k(x_i)=x_i-\frac{1}{2}$ and if $i>k$ then $\Phi_k(x_i)=x_i-2-\frac{1}{2}$. Substituting these $x$-coordinates into \eqref{sgn222}, $(ii)$ immediately follows by using \eqref{hello2}

    \end{proof}
    
We now deal with what we have intentionally excluded from \cref{signsforthefinaltime}; namely, if one of the vertices lies at $x=\frac{1}{2}$ and $k\neq0$. For $k\neq0$, the decoration ($\delta$) of the cup $\mq\in\underline{\mu}$ differs from that of $\mq'\in\underline{\Phi_k(\mu)}$ if and only if $l_{\mq'}=\frac{1}{2}$ meaning our argument above will not hold. We now make the necessary adjustments to deal with this exceptional case.
    
    \begin{lem}
\label{signsanothergo33}
Let $\mu, \la \in \mptn$ be such that $\la=\mu-\cp$ for $\cp\in\underline{\mu}$. Let $\gr,\ot \in \underline{\la}$ be adjacent to $\cp$.
Additionally suppose that $\langle \cp\cup \ot\rangle_\mu$ exists (and hence $\underline{\mu}\mu$ and $\underline{\la}\la$ are as pictured in \cref{adj lem2}).
Let $x_1,x_2,x_3,x_4\in\mathbb{Z}_{\geq0}+\frac{1}{2}$ be the x-coordinates of the vertices of $\cp,\langle \cp\cup \ot\rangle_\mu\in\underline{\mu}$ (and hence of $\gr,\ot\in\underline{\la}$ too).

Additionally, suppose that $\la=\mu-\cp$ is contractible at $k\in\mathbb{Z}_{\geq0}$. 
As above, let $\Phi_k(x_1)<\Phi_k(x_2)<\Phi_k(x_3)<\Phi_k(x_4)\in\frac{1}{2}+\mathbb{Z}_{\geq0}$ be the x-coordinates of the vertices of $\cp',\langle \cp'\cup \ot'\rangle_{\Phi_k(\mu)}\in\underline{\Phi_k(\mu)}$
(and hence of $\gr',\ot'\in\underline{\Phi_k(\la)}$ also). 

Suppose that $\Phi_k(x_1)=\frac{1}{2}$ and $k\neq0$.

$(i)$ Then we have that:
\begin{equation}
\begin{aligned}
\alpha_{x_2}&=\alpha_{\Phi_k({x_2})}\\
\alpha_{x_3}&=\alpha_{\Phi_k({x_3})}\\
\alpha_{x_4}&=\alpha_{\Phi_k({x_4})}\\
\end{aligned}
\end{equation}
$(ii)$ Additionally, we have that $\Phi_k(x_1)=\frac{1}{2}$ and
\begin{equation}
\label{alpha1} 2\alpha_{x_1}+2(\delta_{\text{dec},\ot}- \delta_{\text{dec},\langle \cp\cup \ot\rangle_\mu})=2\alpha_{\Phi_k({x_1})}+2(\delta_{\text{dec},\ot'}-\delta_{\text{dec},\langle \cp'\cup \ot'\rangle_{\Phi_k(\mu)}}).\end{equation}
$(iii)$ Consequently, we then obtain  that $$ C(\cp,\langle \cp\cup \ot\rangle_\mu)-C(\gr,\ot)\equiv C(\cp',\langle \cp'\cup \ot'\rangle_{\Phi_k(\mu)})-C(\gr',\ot')\pmod 8.$$
\end{lem}

\begin{proof}
This proof will naturally follow in the same manner as that of  \cref{signsforthefinaltime}.
First, we note again that the common vertex between the cups $\cp$ and $\ot$ is unchanged under the contraction mapping. The terms in \eqref{hellox222222}-\eqref{hello5555} are independent of decorations and so $(i)$ follows.

To establish $(ii)$, we consider each of the three possible common nodes of the cups $\ot$ and $\cp$ individually and compute the possible values that $\alpha_{x_1}$ and $\alpha_{\Phi_k({x_1}}$ could take in \eqref{sgn222}. Again, we emphasise that even if the decorations change, the common vertex between the cups $\cp$ and $\ot$ is unchanged under the contraction mapping.

 First, if $r_\ot=l_\cp$, then $\alpha_{x_1}=\alpha_{\Phi_k({x_1})}=0$. Here, $\ot$ corresponds to the leftmost cup in $\underline{\la}$ from \cref{adj lem2} $a)$ and $b)$. Since $l_{\ot'}=l_{\langle \cp'\cup \ot'\rangle_{\Phi_k(\mu)}}=\frac{1}{2}$, we obtain $$\delta_{\text{dec},\langle \cp\cup \ot\rangle_\mu}=\delta_{\text{dec},\ot} \quad \text{and} \quad \delta_{\text{dec},\langle \cp'\cup \ot'\rangle_{\Phi_k(\mu)}}=\delta_{\text{dec},\ot'}.$$
 Therefore, \eqref{alpha1} holds as required.

 Second, if $l_\ot=r_\cp$, then $\ot$ corresponds to the rightmost cup in $\underline{\la}$ from \cref{adj lem2} $a)$ and $b)$, meaning $\ot$ is undecorated. 
 Using \eqref{hello111111}, we have that $\alpha_{x_1}=1$ if and only if $\delta_{\text{dec},\ot}\neq \delta_{\text{dec},\langle \cp\cup \ot\rangle_\mu}$ and $\alpha_{x_1}=0$ otherwise. 
It follows that 
 $$2\alpha_{x_1}+2(\delta_{\text{dec},\ot}-\delta_{\text{dec},\langle \cp\cup \ot\rangle_\mu})=2\alpha_{\Phi_k({x_1})}+2(\delta_{\text{dec},\ot'}-\delta_{\text{dec},\langle \cp'\cup \ot'\rangle_{\Phi_k(\mu)}})=0,$$
 and  \eqref{alpha1} holds as required.

 Lastly, if $l_\ot=l_\cp$, then $\ot$ is as pictured in \cref{adj lem2} $c)$ and $d)$. It is easy to check that $$2\alpha_{x_1}+2(\delta_{\text{dec},\ot}-\delta_{\text{dec},\langle \cp\cup \ot\rangle_\mu})=2\alpha_{\Phi_k({x_1})}+2(\delta_{\text{dec},\ot'}-\delta_{\text{dec},\langle \cp'\cup \ot'\rangle_{\Phi_k(\mu)}})=-2.$$ Therefore \eqref{alpha1} holds as required, which completes the proof of $(ii)$.
 
 Finally, we prove $(iii)$. If $\Phi_k(x_1)=\frac{1}{2}$, then either $k=1$ and $x_1=\frac{5}{2}$ or $k>1$ and $x_1=\Phi_k(x_1)=\frac{1}{2}$.
In either case, we applying  $(ii)$ yields that
\begin{equation}
\label{Isthisover}
 4x_1\alpha_{x_1}+2(\delta_{\text{dec},\ot}- \delta_{\text{dec},\langle \cp\cup \ot\rangle_\mu})\equiv4\Phi_k({x_1})\alpha_{\Phi_k({x_1})}+2(\delta_{\text{dec},\ot'}-\delta_{\text{dec},\langle \cp'\cup \ot'\rangle_{\Phi_k(\mu)}})\pmod 8.\end{equation}
 
 As in the conclusion of the proof of \cref{signsforthefinaltime}, let $x_i=i-\frac{1}{2}$ in $\mu$ for $i>1$. If $i<k$ then $\Phi_k(x_i)=x_i-\frac{1}{2}$ and if $i>k$ then $\Phi_k(x_i)=x_i-2-\frac{1}{2}$. Substituting these $x$-coordinates into \eqref{sgn222}, $(iii)$ follows immediately from parts $(i)$ and $(ii)$ together with \eqref{Isthisover}.
    \end{proof}

We have now finished our preparatory lemmas and are ready for the main proof of the final relation.

\begin{prop}[The adjacency relations] 
\label{Kadjproof}
Let $\la=\mu-\cp$ and suppose that $\cp\in \underline{\mu}$ and $\ot\in \underline{\la}$ are adjacent. Set $\sigma= \mu-\cp-\ot=\la-\ot, \alpha =\mu-\langle \cp\cup \ot\rangle_\mu$. Then we have
\begin{equation}
\label{K adj}
i^{\ka(\sigma,\la)+\ka(\la,\mu)}(\underline{\sigma}\sigma\overline{\la})(\underline{\la}\la\overline{\mu})=
\left\{ \begin{array}{ll} i^{\ka(\alpha,\sigma)+\ka(\alpha,\mu)+2(b(\langle \cp\cup \ot\rangle_\mu) - b(\ot))}
(\underline{\sigma}\alpha\overline{\alpha})(\underline{\alpha}\alpha\overline{\mu}) &
 \mbox{if $\langle \cp\cup \ot\rangle_\mu$ exists}
   \\ 0 
& \mbox{otherwise.} \end{array} \right.
\end{equation}
\end{prop}

\begin{proof}
First suppose $\underline{\mu}\la$
 is  contractible at some $k\in \mathbb Z_{\geq0}$. Applying \cref{KContraction homomorphism} and \cref{signsforthefinaltime,signsanothergo33} inductively to \eqref{K adj}, it suffices to restrict our focus only to the incontractible cases.

We first consider when $\cp$ is covered. 
Then $\cp$ is adjacent to two commuting cups, which we labelled $\gr, \ot$ (from left to right). 
Under the assumption that $ \underline{\mu}\la$ is incontractible, we need only consider the case that 
$b(\cp)=1$ and $b(\langle \cp\cup \ot\rangle_\mu)=2$. Note also that $b(\ot)=b(\gr)=1$.
Explicitly this is $\mu=(1,2,\dots,c^k)$, $\la=(1,2,\dots,c^{k-1},c-1)$ and $\alpha=(1,2,\dots,(c-1)^{k-1},c-2)$.
If $c$ is odd and $k=2$, by performing \cref{idempot lem} $\tfrac{c+1}{2}$ times, we get that $(\underline{\sigma}\sigma\overline{\la})(\underline{\la}\la\overline{\mu})$ is equal to
\begin{equation}
\label{adj pic 1}
\begin{minipage}{4.5cm}
 
\end{minipage}\end{equation} 
and 
$(\underline{\sigma}\alpha\overline{\alpha})(\underline{\alpha}\alpha\overline{\mu})$ is equal to
\begin{equation}
\label{adj pic 2}
\begin{minipage}{4.5cm}%
 
\end{minipage}.\end{equation}
Combined with calculating that $\ka(\sigma,\la)+\ka(\la,\mu)=5$ and $\ka(\alpha,\sigma)+\ka(\alpha,\mu)=3$, it follows that \eqref{K adj} holds for $\ot$ adjacent to $\cp$.

We also need to check factoring through $\nu=\la-\gr$. Applying \cref{idempot lem} $\tfrac{c+1}{2}$ times, we have that $(\underline{\nu}\nu\overline{\la})(\underline{\la}\la\overline{\mu})$ is equal to
\begin{equation}
\label{adj pic 22}
\begin{minipage}{4.5cm}%
 
\end{minipage}\end{equation} 
and 
$(\underline{\nu}\alpha\overline{\alpha})(\underline{\alpha}\alpha\overline{\mu})$ is equal to
\begin{equation}
\label{adj pic 3}
\begin{minipage}{4.7cm}%
 
\end{minipage}.\end{equation} 
In this case $\ka(\sigma,\la)+\ka(\la,\mu)=3$ and $\ka(\alpha,\sigma)+\ka(\alpha,\mu)=5$
and so the relation holds independent of the choice of $\ot$ or $\gr$ adjacent to $\cp$.

Notice that decorating the strands doesn't change any of the merging calculations performed above, and as such, $\cp$ being undecorated corresponds to almost identical calculations. 
If $k>2$, $\langle \cp\cup \ot\rangle_\mu$ and $\gr$ will no longer be decorated; $\langle \cp\cup \ot\rangle_\mu$ will also be covered by $k-3$ small anti-clockwise oriented undecorated circles and one decorated small anti-clockwise oriented circle, but due to \cref{idempot lem}, these can be essentially ignored. The powers of $i$ follow similarly.
If $k>c$, there will be strands on the left-hand side of the calculation that can similarly be ignored.
Finally, if $c$ is even, there will be no decorations present, and the calculation is essentially identical to the adjacency relation in \cite[Section 9]{ChrisDyckPaper}. 
Now we treat the case where $\cp$ is doubly covered.
Then $\cp$ is adjacent to two non-commuting cups $\gr \prec \ot$.
Under the assumption that $ \underline{\mu}\la$ is incontractible, we only need to consider the case when $b(\cp)=1$ and $b(\langle \cp\cup \ot\rangle_\mu)=3$. Explicitly, this is $\mu=(1,2,\dots,c)$, $\la=(1,2,\dots,(c-1)^2)$ and $\alpha=(1,2,\dots,c-2,(c-3)^2)$. If $c$ is even, by merging the $\frac{c}{2}$ disjoint circles, using \cref{idempot lem}, we obtain $(\underline{\sigma}\sigma\overline{\la})(\underline{\la}\la\overline{\mu})$ is equal to
\begin{equation}
\label{adj pic 5}
\begin{minipage}{4.5cm}
 
\end{minipage}\end{equation} 
and 
$(\underline{\sigma}\alpha\overline{\alpha})(\underline{\alpha}\alpha\overline{\mu})$ is equal to
\begin{equation}
\label{adj pic 6}
\begin{minipage}{4.5cm}%
 
\end{minipage}.\end{equation}

Once again, this combined with $\ka(\sigma,\la)+\ka(\la,\mu)=3$, $\ka(\alpha,\sigma)+\ka(\alpha,\mu)=5$ and $b(\ot)=2$ means that \eqref{K adj} holds.

However if we take the product with $\nu=\la-\gr$ then $(\underline{\nu}\nu\overline{\la})(\underline{\la}\la\overline{\mu})$ is equal to
\begin{equation}
\label{adj pic 7}
\begin{minipage}{4.5cm}%
 
\end{minipage}=0\end{equation} 
as the resulting diagram is unorientable and so the \eqref{K adj} holds, since $\langle \cp\cup \gr\rangle_\mu$ does not exist. If $c$ is odd, $\cp$ is decorated and $\ot$ is not, but an essentially identical calculation also shows that the relation holds.

Finally, if $\cp$ is adjacent to only one cup, then $\langle \cp\cup \ot\rangle_\mu$ does not exist, and one can easily verify that the surgery procedure always results in an unorientable line as pictured in \cref{adj pic 8} (note that in \cref{adj pic 8}, $\underline{\mu}\alpha$ is not incontractible):
\begin{equation}
\label{adj pic 8}
\begin{minipage}{4.2cm}
 
\end{minipage}=0\end{equation}
\end{proof}
\color{black}

We have now checked that all the relations, \eqref{rel1}--\eqref{adjacentcup}, from \cref{presentation} hold in the image of $\Psi$ and hence we can conclude our proof:

\begin{proof}[Proof of \cref{THE isomorphism}]
By construction, the map $\Psi$ sends the generators of $\mathcal{H}_{(D_n,A_{n-1})}$ to the generators of $\mathbb{D}_n$. Together, \cref{tin pot,KcomR,KCR,KDC,selfdual1,selfdual2,Kadjproof} establish that the Ext-quiver and relations presentation of $\mathcal{H}_{(D_n,A_{n-1})}$ is also satisfied in $\mathbb{D}_n$. Consequently, $\Psi$ is a well-defined and surjective $\mathbb{Z}$-graded homomorphism. Both algebras possess the same graded dimension, as their bases are indexed by the same set of triples $(\la,\nu,\mu)\in\mptn^3$ such that $\underline{\la}\nu$ and $\underline{\mu}\nu$ are oriented cup diagrams (see \cref{cellular basis} and \cref{Khov basis}). Hence, we conclude that $\Psi$ is a $\mathbb{Z}$-graded isomorphism of $\Bbbk$-algebras.
\end{proof}

\appendix
\section{Associativity of the arc algebra}
\label{appendix}

We will show in this appendix   that the arc algebra defined in \cite{TypeDKhov} is not actually associative 
with the signs as defined in their paper (we thank Eliot Grimont for pointing this out and for providing the following low-rank, simple example).  
This is fixed by our choice of signs here.

In \cref{get out of jail remark}, we note that the sole difference in signs in our paper from that of \cite{TypeDKhov}, 
occurs in  \eqref{spliteq}, where we use $(-1)^{\left \lfloor{l_\ps}\right \rfloor }$ (which differs from \cite{TypeDKhov}, where they use $(-1)^{\left \lceil{l_\ps}\right \rceil }$). 
We should highlight that in this paper, we have used the convention that the intersection of an arc (or cup) diagram with the horizontal axis occurs at the half integers, $\{\tfrac{1}{2},\tfrac{3}{2},\tfrac{5}{2}, \dots\}$, whereas in \cite{TypeDKhov} the convention was that the intersections occur at the strictly positive integers $\{1,2,3,\dots\}$. 
The sign in \cite[Section 5.2.2]{TypeDKhov} (the analogue of \eqref{spliteq}) is $(-1)^i$, where $i$ is the leftmost intersection with the horizontal axis of the cup on which surgery is performed; this
  translates to $(-1)^{\left \lceil{l_\ps}\right \rceil }$ in our convention. In what follows, in order
   to make a direct comparison of \cite[Section 5]{TypeDKhov}, we will use their integer convention.

In the following, let $\mu=(1^3), \la=(1,2,1), \nu=(1^2), \sigma=(1,2)$. We picture these weights and their associated cup diagrams in \cref{assoc1}.
\color{black}

\begin{figure}[ht!]
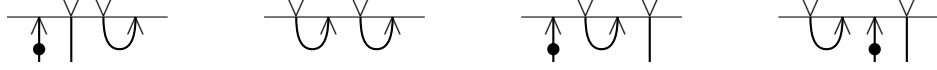


	$$  
$$ 
    
    \caption{$\mu,\la,\nu,\sigma$ pictured from left-to-right}
    \label{assoc1}
    \end{figure}

We then consider the three oriented circle diagrams $D_1=(\underline{\la}\mu\overline{\mu}), D_2=(\underline{\mu}\nu\overline{\nu}), D_3=(\underline{\nu}\nu\overline{\sigma})$ and calculate the product $D_1D_2D_3$ two different ways using the definition of the algebra as given in \cite[Section 5.2]{TypeDKhov}.

First by performing a merge $(3,4)$-surgery and then a split $(2, 3)$-surgery we have that $(D_1D_2)D_3$ is equal to:
\color{black}
\begin{equation}
(D_1D_2)D_3=
 \label{assoc2}
 \begin{aligned}
 & 
 \begin{minipage}{2cm}
 
\end{minipage})\coloneqq-D'.
\end{aligned}
\end{equation} 

\color{black}

However, if we calculate $D_1(D_2D_3)$ by first reconnecting two propgating lines in a  $(2, 3)$-surgery and performing two merges (the leftmost one being trivial) we obtain that:
\color{black}
\begin{equation}
D_1(D_2D_3)=
 \label{assoc2}
 \begin{aligned}
 & 
 \begin{minipage}{2.1cm}
%
 
\end{minipage}=D'.
\end{aligned}
\end{equation} 

 This calculation shows that the algebra as defined in \cite[Section 5]{TypeDKhov} is not associative. 
 This case (and its generalisations to larger ranks) was not considered explicitly in their proof of associativity \cite[Sections 5.3-5.6]{TypeDKhov}, but we remark that it is the only missing case in their proof.  Fixing the sign, as done in our \cref{spliteq}, remedies this small sign error; and the rest of their proof goes through unchanged.

      \bibliographystyle{amsalpha}   
\bibliography{Ben} 
   
\end{document}